\documentclass[11pt]{amsart}
\usepackage{amsmath,amsthm,amsbsy,amssymb}
\usepackage{graphicx}
\usepackage{epstopdf}
\advance \topmargin by-\headheight
\advance \topmargin by-\headsep
\evensidemargin \oddsidemargin
\theoremstyle{plain}

\newtheorem{thm}{Theorem}[section]

\newtheorem{lem}[thm]{Lemma}
\newtheorem*{thm*}{Theorem}
\newtheorem*{lem*}{Lemma}
\newtheorem*{cor*}{Corollary}

\newtheorem*{rem*}{Remark}

\theoremstyle{definition}

\newcommand{\D}{\mathcal{D}}

\newcommand{\F}{\mathcal{F}}

\newcommand{\R}{\mathcal{R}}

\newcommand{\be}{\begin{equation}}
\newcommand{\ee}{\end{equation}}

\renewcommand{\a}{\alpha}
\newcommand{\s}{\sigma}
\renewcommand{\b}{\beta}

\newcommand{\g}{\gamma }
\newcommand{\e}{\epsilon}

\renewcommand{\t}{\tau}
\renewcommand{\Re}{{\rm{Re}}}
\renewcommand{\Im}{{\rm{Im}}}
\renewcommand{\i}{{\mathrm{i}}}    
\renewcommand{\d}{{\mathrm{d}}}

\begin{document}
\title{Finite Euler Products and the Riemann Hypothesis}

\author{S. M. Gonek}

\address{Department of Mathematics, University of Rochester, 
Rochester, NY 14627, USA}
 \email{gonek@math.rochester.edu}

\thanks{This work was supported by NSF grant DMS 0201457 and by an
NSF Focused Research Group grant (DMS 0244660). }

\date{April 26, 2007}

\begin{abstract}
We  show that  if the Riemann  Hypothesis is true,  then in a region containing most 
of the right-half of the critical strip, the Riemann zeta-function is well approximated  
by short truncations of its Euler product. Conversely, if the approximation by  products
is good in this region,  the zeta-function has at most  finitely many zeros in it.  
We then construct a parameterized family of non-analytic functions with this same property.  
With the possible exception of a finite number of zeros off the critical line, every 
function  in the family satisfies a Riemann Hypothesis. Moreover, when  
the parameter is not too large, they have about the same number of zeros as the 
zeta-function, their zeros are all simple, and they ``repel''.  The  structure of  these  
functions makes the reason for the simplicity and repulsion of  their zeros  apparent and 
suggests a mechanism that might be responsible for the corresponding properties 
of the zeta-function's zeros.  Computer evidence suggests that the  zeros of functions in the 
family are remarkably close  to those of the zeta-function (even for small values of the parameter), 
and we show  that they indeed  converge to them as the parameter increases. 
Furthermore, between zeros of the zeta-function, the moduli of   functions in the family tend to twice the 
modulus of the zeta-function. Both   assertions assume the Riemann Hypothesis. We end by discussing analogues 
for other L-functions and show how they give insight into the study  of the distribution of zeros of linear 
combinations of L-functions.  
\end{abstract}

\maketitle

\tableofcontents
\section{Introduction}
 
Why  should the Riemann Hypothesis be true?  If all the  zeros of the zeta-function 
are simple, why?  Why do the zeros seem to repel each other?  Analytic number theorists  
believe  that an eventual proof of the Riemann Hypothesis must use both the Euler product 
and functional equation of the zeta-function.  For there are functions with similar functional 
equations  but  no Euler product,  and    functions with  an Euler product  but no functional 
equation, for which the Riemann Hypothesis is false.  But \emph{why}  are these two ingredients  
essential?

This paper began as an attempt to gain insight into these questions.   

Section 2 begins with a brief discussion of the approximation of $\zeta(s)$
by truncations of its Dirichlet series. In Section 3 we turn to the approximation  
of $\zeta(s)$  by truncations of its Euler product. We show that if the Riemann Hypothesis
is true, then short  products approximate the zeta-function well in a region containing most 
of the right-half of the critical strip. Conversely,  if  the approximation 
by products is good in this region, the zeta-function has at most  finitely many zeros in it.  
Section 4 is a slight departure from the  main direction of the paper, but  we include it in 
order  to deduce some imediate  consequences  of the results of Section 3.  
In Section 5  we construct a parameterized family,  $\{\zeta_X(s)\}$,  
of functions related to the zeta-function with the same type of approximation property
as  the finite Euler products. That is, if the Riemann Hypothesis is true, then $\zeta_X(s)$ is a good 
approximation of $\zeta(s)$ in a region containing most of the right-half of the critical strip and, 
if $\zeta_X(s)$ is a good approximation of $\zeta(s)$ in this region, then $\zeta(s)$ can have at most
 finitely many zeros there.
In Section 6 we show that, with the possible exception of a few  low lying zeros,  
a Riemann Hypothesis   holds for each $\zeta_X(s)$.  In Sections 7 and 8, respectively, we prove 
that on the Riemann Hypothesis, if the parameter $X$  is not too large,  then  $\zeta_X(s)$ has about 
the same number of zeros as $\zeta(s)$ and  that its zeros are all simple (again with the possible 
exception of a few  low lying ones). We also show unconditionally that when the parameter 
is much larger, $\zeta_X(s)$ still has asymptotically the same number of zeros as $\zeta(s)$ 
and that $100\%$ of these  (in the density sense) are simple. In the next section we study the 
relationship between the  two functions on the critical line. Assuming the Riemannn Hypothesis, 
we  show that the zeros of  $\zeta_X(s)$ converge to the zeros of $\zeta(s)$ as $X \to \infty$
and that between the zeros  of the zeta-function $|\zeta_X(\frac12+it)| \to  2 |\zeta(\frac12+it)|$. 
In Section 10 we suggest 
possible causes for the simplicity  and  repulsion of the zeros of $\zeta(s)$ in light of 
the structure of   $\zeta_X(s)$. In the last section we illustrate how our results generalize to 
other L-functions by defining functions $L_X(s, \chi)$ corresponding to the Dirichlet L-function
$L(s, \chi)$. We then  study the distribution of  zeros of linear combinations of 
$L_X(s, \chi)$. This suggests a  heuristic  different from the usual one (of carrier waves) 
for understanding why linear combinations of the standard L-functions should have $100\%$
of their zeros on the critical line.  The  appendix provides some useful approximations of the 
zeta-function by Dirichlet polynomials.

The functions $\zeta_X(s)$ are simpler than the Riemann zeta-function,  
yet they  capture  some of its most  important structural features. 
It  therefore makes sense to regard them as models  of the zeta-function.  
The modeling is probably best when $X$ is large,  say a power of $t$. 
Unfortunately, our  results   are most satisfactory  only for somewhat smaller 
$X$ ranges, so it would be  interesting if one could  extend them.

We began by raising several deep questions. We do not offer definitive 
answers here, just ones that might be suggestive.
For example, we shall show
that most of the zeros of $\zeta_X(s)$ are  simple 
and we shall see a mechanism that causes many of them, if not most, to repel. 
In fact, we shall prove that \emph{all} the zeros  of $\zeta_X(s)$ (above a certain height) 
are simple and repel if $X$ is not too large. Since  $\zeta_X(s)$ mimics $\zeta(s)$,  this suggests 
that the zeros of  the latter should share these properties. And
concerning the question of  why both the Euler product and functional equation are 
necessary for the  truth of the Riemann Hypothesis, one possible interpretation of our 
results is this:  the Euler product prevents the zeta-function from having zeros off the line, 
the functional equation puts them on it. 
  
I wish to thank Enrico Bombieri, Dimitri Gioev, Jon Keating, and Peter Sarnak for  very helpful conversations and communications. I owe an especially great debt to  Dimitri Gioev for our  stimulating discussions
over many months and for the extensive computer calulations he performed.


\section{The approximation of  $\zeta(s)$ by finite Dirichlet series}
\label{series approx}

Throughout we write $s=\s+it$  and $\t=|t|+2$; $\e$ denotes an arbitrarily small positive 
number which may not be the same at each occurrence. 

The Riemann zeta-function is analytic in the entire complex plane, 
except for a simple pole at $s=1$ and in the half-plane $\s >1$ it is given by
the   absolutely convergent series
\begin{equation*}\label{zeta sum defn}
\zeta(s) = \sum_{n=1}^{\infty}  n^{-s}\,.
\end{equation*}
Estimating the tail of the series trivially, we obtain  
the approximation
\begin{equation}\label{zeta approx 1}
\zeta(s) = \sum_{n=1}^{X} n^{-s} + O\left(\frac{X^{1-\s}}{\s-1} \right)
\end{equation}
for $\s >1$ and $X \geq 1$.  
A crude form of the approximate functional equation  (see Titchmarsh~\cite{T})
extends this into the critical strip:
 \begin{equation}\label{zeta approx 2}
\zeta(s) = \sum_{n=1}^{X} n^{-s} +  \frac{X^{1-s}}{s-1}
+ O(X^{-\s})\,.
\end{equation}
This holds  uniformly for $\s \geq\s_0>0$, provided that $X \geq C\; \t /2\pi$, 
where $C$ is any constant greater than $1$. The second term on the right-hand
side reflects the simple pole of $\zeta(s)$ at $s=1$ and, if we stay away from it,
it can be ignored. For instance, setting $X = t$ and assuming $t \geq 1$,  
we find that
\begin{equation}\label{zeta approx 3}
\zeta(s) = \sum_{n\leq  t } n^{-s} + O(t^{-\s})
\end{equation}
uniformly for $\s \geq\s_0>0$. 
Thus, truncations of the Dirichlet series defining
$\zeta(s)$ approximate it well, even in the critical strip. 

Now suppose  that    the Lindel\"{o}f Hypothesis is true. 
That is, that
\begin{equation*}\label{LH}
\zeta(\tfrac12 + \i t) \ll  \t^{\epsilon}.   
\end{equation*}
Then the length of the series in \eqref{zeta approx 2}
and \eqref{zeta approx 3} can be considerably reduced,
as the following modification of Theorem 13.3 of 
Titchmarsh~\cite{T} shows. (See the Appendix.)

\begin{thm}\label{thm on Lindelof 1}
Let   $\s$ be bounded, $ |\sigma| \geq \frac12$,
and $|s-1|> \frac{1}{10}$.   
Also let $  1 \leq X  \leq t^{2}$.
A necessary and 
sufficient condition for the truth of the Lindel\"{o}f Hypothesis is that   
\begin{equation*}\label{Lindelof approx 4}
 \zeta(s) = \sum_{n \leq X} \frac{1}{n^s}   
+  O\left(  X^{ \frac12 - \sigma}\t^{\epsilon}   \right) \,.
\end{equation*}
 \end{thm}

It follows from this that 
if the Lindel\"{o}f Hypothesis is true and we  stay away 
from the pole of the zeta-function at $s=1$, then
$\zeta(s)$ is well approximated by arbitrarily short truncations of its 
Dirichlet series in the half plane $\s > \frac12 $. Of course, we saw that this is unconditionally 
true in the half plane $\s > 1$. 

On the other hand, short sums can \emph{not} 
approximate  $\zeta(s)$ well in the strip  $0<\s \leq 1/2$. 
For suppose that such a sum and $\zeta(s)$ were within 
$\e$ of each other, where $\e>0$ is small. Then  we would have
\begin{equation}\label{difference ineq}
\int_{T}^{2T} \left| \zeta(s)   -  \sum_{n\leq X}n^{-s} \right|^2 dt
\leq \e^2 T\,.
\end{equation}
However, if $0< \s \leq \frac12$ is fixed and $X<T^{1-\e}$, then  
$$
\int_{T}^{2T} \left| \sum_{n\leq X}\frac{1}{n^{s}} \right|^2 dt \sim
\begin{cases} T \left(\frac{X^{1-\s}-1}{1-\s}  \right)  & \hbox{ if } \s < \frac12, \\
T \log X  & \hbox{ if } \s=\frac12  \end{cases}
$$
and  
$$
\int_{T}^{2T} \left| \zeta(\s + i t) \right|^2 dt  \sim 
\begin{cases} C(\s) T^{2-2\s}  & \hbox{ if } \s < \frac12, \\
T \log T  & \hbox{ if } \s=\frac12 \;. \end{cases}
$$
Comparing these when $ \s < \frac12$ and again when $\s = \frac12$, 
we obtain a contradiction to  \eqref{difference ineq}.
 This argument is unconditional  
 and shows that one cannot do   better then 
\eqref{zeta approx 3} even if the Lindel\"{o}f  or Riemann Hypothesis is true.

To summarize, $\zeta(s)$ is  well-approximated  unconditionally by arbitrarily short truncations of its
Dirichlet series in the region $\s>1$, $|s-1|>\frac{1}{10}$.  
On the Lindel\"{o}f Hypothesis 
this remains true even in the right-half of the critical strip, $\frac12<\s\leq1$. However, 
on and to the left of the critical line, the length of the truncation must be  
 $\approx t$. The situation is the same if we assume  the 
Riemann Hypothesis instead of  the Lindel\"{o}f Hypothesis, since the former implies the latter.


\section{The approximation of  $\zeta(s)$ by finite Euler products}\label{product approx}

The zeta-function also has  the   Euler product representation
\begin{equation*}
\zeta(s)= \prod_{p} \left(  1 - \frac{1}{p^s} \right)^{-1}  
\end{equation*}
in the half plane $\s>1$,
where the product is over all  prime numbers. This converges absolutely
 and it is straightforward to show (take logarithms)
that
\begin{equation}\label{prod 1}
\zeta(s)= \prod_{p \leq X} \left(  1 - \frac{1}{p^s} \right)^{-1} 
\left( 1+ O\left(\frac{X^{1-\sigma}}{(\s-1)\log X} \right) \right) ,
\end{equation}
for $\s > 1 $.
Here we implicitly use the fact that $\zeta(s)$ 
does not vanish in $\s>1$.
As is often the case, it is more natural from an analytic point of view to work with   
weighted approximations, so we will use  expressions  of the type 
\begin{equation*}\label{eq:defn_P}
 \exp\left(\sum_{n }\frac{\Lambda(n) v(n)}
 {n^{s}\log n} \right),
\end{equation*}
where $\Lambda(n)$ is von Mangoldt's function and  the weights  $v(n)$ 
will be specified later. 
 
We next ask whether  it is possible to extend  \eqref{prod 1}  (or 
a weighted form of  it) into the critical strip in the same way that \eqref{zeta approx 2}  
extended \eqref{zeta approx 1}. A recent result of  Gonek, Hughes, and Keating~\cite{GHK}  
suggests an answer.   It says   that if   $X < t^{1-\e}$ and  $X$ is not too small, then
$\zeta(s)$ factors  in the region $\sigma \geq 0$,  $|s-1| > \frac{1}{10}$ 
as
\begin{equation}\label{hybrid}
 \zeta(s) =  \exp\left(\sum_{n\leq X }\frac{\Lambda(n)  }
 {n^{s}\log n} \right) Z_{X}(s) \big( 1+o(1) \big) \,,
\end{equation}
where $Z_X(s)$ is a certain product over the zeros of $\zeta(s)$.
Now, one can show that in the right-half of the critical strip $Z_{X}(s)$ is close to  $1$
as long as $s$ is not too  near a zero of $\zeta(s)$.   
Hence,  if  the Riemann Hypothesis is true and $s$ 
is not too close to the  critical line, $Z_{X}(s)$ will be close to  $1$. 
(The closer $\s$ is to $\frac12$,  the larger one needs to take $X$.) Thus, under the 
Riemann Hypothesis, an analogue of \eqref{prod 1} does hold in the 
right-half of the critical strip. 

To prove these assertions we  need an explicit version of \eqref{hybrid}  that
differs slightly from the one given by Gonek, Hughes and Keating~\cite{GHK}, 
and we derive this next.  

Write
 \begin{equation}\label{P}
P_{X}(s) = \exp \left( \sum_{n\leq X^2} 
\frac{\Lambda_X(n)}{n^{s} \log n}  \right),
\end{equation}
where 
\begin{equation*}
\Lambda_X(n) = 
\begin{cases} \Lambda(n) & \hbox{ if } n \leq X, \\
\Lambda(n)\left(2- \frac{\log n}{\log X}\right)& \hbox{ if } X< n \leq X^2, \\
0 & \hbox{ if } n > X^2 \;. 
\end{cases}
\end{equation*}
Also, let
\begin{equation}\label{E_2 defn}
E_2(z) = \int_{z}^{\infty}\frac{e^{-w}}{w^2} dw \qquad (z\neq0)
\end{equation}
denote the second exponential integral,  and set
\begin{equation*}\label{F_2 defn}
F_2(z) = 2E_2(2z) - E_2(z)\,.
\end{equation*}
We then define
\begin{equation}\label{Z}
  Z_X(s) =   \exp \bigg(  \sum_{\rho}  
 F_2\left((s-\rho)\log X \right)   
 - F_2\left((s-1)\log X \right)   \bigg) \,,
\end{equation}
where the sum is over all the non-trivial zeros, $\rho=\beta+i \g$, of the zeta-function.

Our variant of \eqref{hybrid} is
\begin{thm}\label{thm on zeta}  
Let $ \s \geq 0$  and $  X \geq 2$. With 
$P_X(s) $ and $Z_X(s) $  as above we have
\begin{equation}\label{zeta form 1}
\zeta(s) = P_{X}(s) Z_X(s)
 \left(1+  O\left(  \frac{X^{-\s-2  }  }{\t^2 \log^2  X}  \right) \right)\,.
\end{equation}
\end{thm}
\begin{proof}
We   begin with the explicit formula (see Titchmarsh~\cite{T}, Theorem 14.20)
\begin{align*}
\frac{\zeta^{'}}{\zeta}(s) 
=   - \sum_{n\leq X^2} \frac{\Lambda_X(n)}{n^s} \notag
&+ \frac{X^{2(1-s)} - X^{(1-s)}}{(1-s)^2  \log X} 
+   \sum_{\rho} \frac{X^{\rho-s} - X^{2(\rho-s)}}{(s-\rho)^2 \log X    }  \\
& \qquad + \frac{1}{\log X} \sum_{q=1}^{\infty}  
\frac{X^{-( 2q+s)} - X^{-2(2q+s)}}{(s+2q)^2  }    \,.
\end{align*}
The last term on the right is easily seen to be
$\ll
  X^{-\s-2}/\t^2  \log X $, so  we have
\begin{align}\label{formula 1}
\frac{\zeta^{'}}{\zeta}(s) 
=   - \sum_{n\leq X^2} \frac{\Lambda_X(n)}{n^s}  
+ \frac{X^{2(1-s)} - X^{(1-s)}}{(1-s)^2  \log X} +    
 \sum_{\rho} \frac{X^{\rho-s} - X^{2(\rho-s)} }{(s-\rho)^2 \log X } 
+ O\left( \frac{ X^{-\s-2}}{\t^2  \log X} \right)\,.
\end{align}

Next we integrate \eqref{formula 1} from  $\infty$ to  $s_0= \s_0+i t$,
where $\s_0+i t$ is not a 
zero of the zeta-function. We use the convention that if $t$ is the ordinate of a zero
$\rho=\beta +i \gamma$ and   $0\leq  \s_0 < \beta$, then 
\begin{equation}\label{convention}
\log \zeta(\s_0+i \g) =   \lim_{\e \to 0^+}  \log \zeta \left(\s_0+i (\g +\e ) \right)  \,.
\end{equation}
The $O$-term  contributes
$$
\ll \frac{ X^{  -\s_0-2  } }{\t^2  \log^2 X}  \;.
$$
We also see that
 $$
 \int_{\infty }^{\s_0+i t} \frac{Y^{\rho-s} -Y^{2(\rho-s) } }{(s-\rho)^2  } ds
 =  \log Y \,F_2\big((s_0-\rho)\log Y\big)\,.
 $$
Here we use the  convention analogous to \eqref{convention} if $t$
is the ordinate of a zero.
It follows that 
\begin{align}\label{log zeta}
\log \zeta(s_0) 
=    \sum_{n\leq X^2} \frac{\Lambda_X(n)}{n^{s_0} \log n} \notag
+       \sum_{\rho} F_2\left((s_0-\rho)\log X \right) 
- F_2\left((s_0-1)\log X \right)   
+ O\left( \frac{X^{ -\s_0 -2 }  }{\t^2 \log^2 X} \right)\,.
\end{align} 
Replacing $\s_0$ by $\s$ and exponentiating both sides,  we obtain the stated result
when $s$ is not equal to a zero $\rho$. If it is, we may interpret the factor in
 \eqref{Z}   corresponding to  $\rho$ as 
$
\lim_{\e \to 0^+}\, \exp \left(
F_2\left(i \e \log X \right)
 \right).
$
 From the  well known formula
$E_2(z) = 1/z + \log z + e_2(z)$,
where $ |\arg z|< \pi $ and $e_2(z)$ is analytic in $z$, it follows that
\begin{equation}\label{F_2 formula 1}
 F_2(z) =  \log 4z + f_2(z)      \qquad (|\arg z|<\pi), 
\end{equation}
where $f_2(z) $ is also analytic.
Therefore  
$
\lim_{\e \to 0^+}\, \exp \left(
F_2\left(i \e \log X \right)
 \right) =0\,.
$
Since this agrees with the left-hand side of \eqref{zeta form 1},
 the formula is valid in this case as well. 
\end{proof}

Before stating our next result we require some notation and a lemma.    
As usual we write  $S(t) =(1/\pi )\arg \zeta(\frac12+it)$ with the convention that 
if $t$ is the ordinate of a zero,  
$
S(t) =   \lim_{\e \to 0^+} S(t+\e)   .
$ 
For $t \geq 0$ we  let $\Phi(t)$ denote  a positive increasing differentiable 
function such that 
\begin{equation*}\label{S and Zeta bound}
|S(t)| \leq \Phi(t) \qquad \hbox{and} \qquad  
 |\zeta(\frac12+i t)| \ll \exp(\Phi(t)) \,, 
\end{equation*} 
and  such that for $t$ sufficiently large we have
\begin{equation}\label{Phi}
\Phi^{'}(t)/\Phi(t) \ll \frac1{t \log t }\,.
\end{equation} 
We call such a function \emph{admissible}.
Note that any function of the type $f(t)=(\log\t)^\a (\log\log3\t)^\b$
with $\a$ positive satisfies \eqref{Phi}. Furthermore, it is easily checked that 
if $\Phi$ satisfies  \eqref{Phi}, then 
\begin{equation}\label{Phi ineq}
\Phi(t^a) \ll \Phi(t)\,,
\end{equation}
where the implied constant depends at most on $a$.
It is known that  $\Phi(t)= \frac16\log \t$ is admissible and, 
on the Lindel\"{o}f Hypothesis, that $\e \log \t$ is for any $\e>0$.  
If the Riemannn Hypothesis is true, then $\Phi(t)= \frac12\log \t/\log\log 2\t$ is 
 admissible. (The constant $\frac12$ is  a recent result due to Goldston and 
Gonek~\cite{GG}.)
Balasubramaian and Ramachandra~\cite{BR} (see also Titchmarsh~\cite{T}, 
pp.208-209 and p. 384) have shown that  if $\Phi$ is admissible,   
$\Phi(t)= \Omega(\sqrt{\log \t/\log\log 2\t}\,)$, and this is  unconditional.
Farmer, Gonek and Hughes~\cite{FGH} have conjectured
 that $\Phi(t)= \sqrt{(\frac12+\e)\log \t \log\log 2\t}$ is admissible, but   
 $\Phi(t)= \sqrt{(\frac12-\e)\log \t\log\log 2\t}$ is not.

For the remainder of this paper $\Phi$ will always denote an admissible function.

We can now state our lemma.
\begin{lem}\label{Zero Sum} Assume the Riemann Hypothesis.
Suppose that $\Phi(t)$ is admissible and that 
 $\s >\frac12 $ is bounded. Then we have
\begin{equation}\label{fund sum}
\sum_{\g} \frac{\s-\frac12}{(\s-\frac12)^2+(t-\g)^2} \ll  \log \t 
+ \frac{\Phi(t)}{\s-\frac12} \,.
\end{equation}
Moreover, if $  \Delta >0$, then
\begin{equation}\label{fund sum 2}
\sum_{|\g-t| > \Delta} \frac{1}{ (t-\g)^2} 
\ll  \frac{1}{\Delta} \left(\log \t  + \frac{\Phi(\t)}{\Delta} \right) \,.
\end{equation}
\end{lem}                           
\begin{rem*} With more care we could show that the first sum 
equals $\frac12 \log \t + O\left( \Phi(\t)/(\s-\frac12) \right)$, but we 
do not require this.
\end{rem*}

\begin{proof}
For the sake of convenience we write $\s-\frac12 = a$. 
Recall that $N(t),$ the number
of zeros of $\zeta(s)$ with ordinates in $[0, t]$, is
\begin{equation*}\label{zero counting formula}
N(t) = \frac{t}{2\pi}\log \frac{t}{2\pi} - \frac{t}{2\pi} 
+\frac78 +S(t) + O\left( \frac{1}{\t} \right) \,.
\end{equation*}
Therefore
\begin{align}\label{zeros in interval}
 N(t+a)-N(t-a)  
= \frac{a}{\pi} \log \frac{t}{2\pi} + O(\Phi(\t ))  \,. 
\end{align}
 The  left-hand side of 
\eqref{fund sum} is 
\begin{equation*}\label{intermediate sum}
\ll  \sum_{0\leq k \leq 1/a} \quad
\left( \sum_{k a \leq |\g-t| \leq (k+1)a } 
\frac{a}{a^2+(t-\g)^2} \right) 
+ \sum_{|\g-t|>1} \frac{a}{a^2+(t-\g)^2} \,.
\end{equation*}
Using \eqref{zeros in interval}, we see that the  second sum is $\ll a \log \t$ 
and, for each $k$, that the sum in parentheses is
\begin{align*}
 \ll  \big(a\log \t + \Phi(\t )\big) \frac{a}{a^2+(ka)^2} 
 \ll   \frac{1}{1+ k^2} \left( \log \t + \frac{\Phi(\t )}{a} \right) \,. 
\end{align*}   
Summing our estimates, we obtain  \eqref{fund sum}.

The proof of \eqref{fund sum 2} is  similar.    

\end{proof}

Our approximation of the zeta-function by finite Euler products will 
follow almost immediately from
 
\begin{thm}\label{thm on RH}  
Assume the Riemann Hypothesis.  Let    
 $ \s \geq \frac12 + \frac{1}{\log X} $ and $|s-1|\geq \frac{1}{10}$.
Then for any $X \geq 2 $ we have   
\begin{equation}\label{zeta form 4}
\zeta(s) = P_{X}(s) e^{ R_X(s)}\,,
\end{equation}
where
\begin{equation}\label{R_X}
 R_X(s) \ll   X^{\frac12-\s }  
\left(  \Phi(\t)  +\frac{\log \t}{\log X}\right)  
+ \frac{X}{\t^2 \log^2X}   \,.
\end{equation}
Moreover, throughout the region
$\sigma \geq \frac12$, $|s-1|\geq \frac{1}{10}$, 
\begin{equation} \label{arg zeta}
\arg \zeta(\s+ i t) = - \sum_{n\leq X^2} 
  \frac{\Lambda_X(n) \sin(t \log n)}{n^{\s} \log n} 
 + O \left(  R_X(s)  \right)     \,.
 \end{equation}
\end{thm}
 \begin{proof}
We estimate 
\begin{equation}\label{Z formula again} 
  Z_X(s) =   \exp \bigg(  \sum_{\rho}  
 F_2\left((s-\rho)\log X \right)  - F_2\left((s-1)\log X \right)   \bigg)   
\end{equation}
in \eqref{Z}.

Integrating \eqref{E_2 defn} by parts, we see that  
for $|z|\geq 1$  
\begin{equation}\label{E_2 formula 2}
E_2(z)  =  \frac{e^{-z}}{z^2}\left( 1 + O( |z|^{-1})  \right)  \,,
\end{equation}
and therefore that
\begin{equation}\label{F_2 formula 2}
F_2(z)  \ll    e^{\max( -\Re\,z,\, -\Re\,2z)} /  |z|^{2}  \,.
\end{equation}
Since $ \s \geq \frac12 + \frac{1}{\log X }$ and 
 the zeros are  of the form
$\rho=\frac12 + i \gamma$, they all
satisfy $|s-\rho|\log X  \geq 1$.
Thus, by \eqref{F_2 formula 2} and Lemma~\ref{Zero Sum},  
 the sum  in  \eqref{Z formula again}   is
\begin{equation}\label{inequality 1}
    \ll     \frac{1}{\log^2 X}   \sum_{|s-\rho|\log X  \geq 1}
\frac{ X^{\frac12-\s }  }{( \s-\frac12)^2 +(t-\gamma)^2  }    
    \ll    X^{\frac12-\s }   
 \left( \Phi(\t) + \frac{\log \t}{\log X }\right) \,.  
\end{equation}
Also by \eqref{F_2 formula 2},
\begin{equation}\label{F_2 of pole}
F_2\left((s-1)\log X \right)  \ll 
\frac{X^{\max \left( 1-\s,\;2(1-\s) \right)  }  }{\t^2 \log^2  X} \,.
\end{equation}
The  first assertion of the theorem follows from this and  \eqref{zeta form 1}. 

The second assertion 
follows immediately from  \eqref{zeta form 4}  if    
$\s  \geq  \frac12 + \frac{1}{\log X} $, so  
 we need only  consider the case 
$\frac12 \leq \s  < \frac12+ \frac{1}{\log X} $. The terms in the   
sum in \eqref{Z formula again} for which $ |s-\rho|\log X  \geq 1$  contribute the 
same amount as before. However,  now there may also 
be a finite number of terms for which $ |s-\rho|\log X  \leq 1$.
Using \eqref{F_2 formula 1} to estimate these, we find 
that if  $s$  is not a zero,  
 they contribute 
\begin{equation*}
\sum_{|s-\rho|\log X  \leq 1} \bigg( \log (4 (s-\rho) \log X ) + O(1)   \bigg) \,.
\end{equation*}
Since
$|\arg (s-\rho) \log X  | \leq \pi/2 $, the imaginary part of this is 
\begin{align*}
\ll     \sum_{| t-\gamma |  \leq 1/ \log X  } 1   \ 
 \ll   \frac{\log \t}{\log X} + \Phi(t) \,,
\end{align*}
by  \eqref{zeros in interval}. 
This is big-$O$ of the bound in
 \eqref{inequality 1} because $\frac12 \leq \s  < \frac12+ \frac{1}{\log X} $.
Thus, we obtain \eqref{arg zeta} provided that $t$ is not the ordinate of a zero.
 If it is, the result  follows from our convention that
 $\arg \zeta(\s +i t) = \lim_{\e \to 0^{+}}\arg \zeta(\s +i (t+\e))$.
This completes the proof of the theorem.
\end{proof}
 
We can now deduce an approximation of $\zeta(s)$ by  Euler products.

\begin{thm}\label{Approx by P} 
Assume the Riemann Hypothesis.
Let  $|s-1|\geq \frac{1}{10}$ and $ \exp(\log \t /\Phi(t)) \leq X \leq \t^2$.
Then if $ \frac12 +   \frac{ C \log \Phi(t) }{\log X} \leq \s \leq 1$ with $C >1$,
we have
\begin{equation}\label{approx by P1}
\zeta(s) = P_X(s) \left(    1 + O \left(  \Phi(t)^{1-C}  \right)  \right) \,. 
\end{equation}
If $2 \leq X < \exp(\log \t /\Phi(t))$ and $ \frac12 +   \frac{C \log\log 2\t }{\log X} 
\leq \s \leq 1$ with $C >1$, then 
\begin{equation*}\label{approx by P2}
\zeta(s) = P_X(s) \left(    1 + O \left(  (\log  \t)^{1-C}  \right)  \right) \,. 
\end{equation*}
\end{thm}
\begin{proof} 
We estimate $R_X(s)$ in \eqref{zeta form 4}.
First assume that $ \exp(\log \t /\Phi(t)) \leq X \leq \t^2$
and $ \frac12 +   \frac{ C \log \Phi(t) }{\log X} \leq \s \leq 1$ .
Then $\log \t /\log X \leq  \Phi(t)$ and
\begin{align*}
 R_X(s) \ll  & X^{ \frac12-\s }  \Phi(t)  \notag
+ \frac{X^{ 2(\frac12-\s)}}{  \log^2X}  \\
\ll  &\Phi(t)^{1-C}+\Phi(t)^{-2C}   \notag   \\
\ll  &\Phi(t)^{1-C} \,.
\end{align*}
It follows that $\exp(R_X(s)) =1+ O(\Phi(t)^{1-C}) $, so 
we have \eqref{approx by P1}.
The second assertion follows similarly, except that this time $ \Phi(t) < \log \t/\log X$.
\end{proof}

Thus, on the Riemann Hypothesis  short Euler products
approximate   $\zeta(s)$
as long as we are not  too close to the critical line. 


We can combine the two assertions of Theorem~\ref{Approx by P} and prove a partial converse as well.

\begin{thm}\label{Approx by P2} 
Assume the Riemann Hypothesis.  Let  
$2 \leq X \leq t^2$,  $|s-1|\geq \frac{1}{10}$,
and  $ \frac12 +   \frac{ C \log\log 2\t }{\log X} \leq \s \leq 1$ with $C >1$.
Then  
\begin{equation}\label{approx by P3}
\zeta(s) = P_X(s) \left(    1 + O \left( \log^{(1-C)/2} t \right)  \right) \,. 
\end{equation} 
Conversely, if \eqref{approx by P3}   holds  for $2 \leq X \leq t^2$
in the region stated, then $\zeta(s)$ has at most a 
finite number of  zeros to the right of  
$\s= \frac12 +   \frac{ C \log\log 2\t }{\log X}$.
\end{thm}

\begin{rem*} 
The condition  
$\s \geq  \frac12 +   \frac{ C \log\log 2\t }{\log X}$
implies a lower bound for $X$ that grows with $t$, namely,
$$
X \geq (\log \t)^{C/(\s-\frac12)}\,.
$$
\end{rem*}

The converse follows from the observation that if \eqref{approx by P3}
holds, then there is a constant $B>0$ such that 
$$
\left|  \zeta(s)P_X(s)^{-1}  -  1  \right|\leq B \log^{(1-C)/2} \t \,. 
$$
If $\zeta(\b+i\g)=0$ with $\b >\frac12 +   \frac{ C \log\log (2|\g| +2)}{\log X}$, this forces 
$\g \leq \exp(B^{2/(C-1)})$ and the result follows.  

As in the case of  approximations  by short sums,
one can also ask  whether   short products    approximate $\zeta(s)$ 
well when $0 < \s \leq \frac12$.  For sums we saw that the answer is no
unless they are of length at least $t$. For products the answer is no  
no matter how long they are. 
A quick way to see this is by counting zeros of 
$\zeta(s)$ and of $P_X(s)$ in a rectangle containing the segment 
$[\frac12 , \frac12+    iT].$
The former has  $\sim (T/2\pi) \log T$ zeros, the latter none. 
This would be impossible if $\zeta(s) =P_X(s)(1+o(1))$ 
 in the rectangle.

One can also argue as follows when $\s$ is strictly less than $\frac12$.
(A modification of the argument works for $\s=\frac12$ too.)
Suppose that   
$\zeta(s) = P_X(s) (1+ o(1))$ in the strip $0 < \s \leq \frac12$. 
Then   $\log |\zeta(s)| = \log |P_X(s)| + o(1)$ and we have
 \begin{equation}\label{int zeta P1}
 \int_{0}^{T} \left( \log |\zeta(\s + i t) |  \right)^2  dt  \sim 
  \int_{0}^{T} \left( \log |P_X(\s + i t) |  \right)^2  dt  
\end{equation}
for   $\s$ fixed and $T \to \infty$.
By the functional equation for the zeta-function,
\begin{align*}
\log |\zeta(\s+it)| = (\tfrac12-\s)\log\frac{\t}{2\pi} + 
\log  |\zeta(1-\s - it)| + o(1)\,.
\end{align*} 
Now the mean-square of the three terms on the right-hand side are 
  $\sim (\tfrac12-\s)^2 T\log^2 T , \; \sim c_0 T$ and \; $o(T)$, respectively.
Thus, 
\begin{equation}\label{int zeta P2}
 \int_{0}^{T} \left( \log |\zeta(\s + i t) | \right)^2    dt 
\sim (\tfrac12-\s)^2 T\log^2 T \,.
\end{equation}
On the other hand, by the mean value theorem for Dirichlet polynomials, 
 if $X =o(T^{\frac12})$,   the right-hand side 
of \eqref{int zeta P1} is
\begin{align*}
  \int_{0}^{T} \left( \log |P_X(\s + i t) |  \right)^2  dt  
  \sim & \int_{0}^{T} 
\left( \sum_{n\leq X^2} \frac{\Lambda_{X}(n)
\cos(t\log n)}{n^{\s } \log n}  \right)^2 dt     \\
 &\sim \frac{T}{2} \sum_{n\leq X^2} 
 \frac{\Lambda^{2}_{X}(n) }{ n^{ 2\s  } \log^2 n}    \\
& \sim c \,T \frac{X^{2-4\s}}{ \log X  } \,,
\end{align*}
where $c$ is a positive constant. 
Comparing this with \eqref{int zeta P2}, we see
that  \eqref{int zeta P1} cannot hold if $0\leq \s < \frac12$
and  $X$ is larger than a certain power of $\log T$.
Note also that for
 infinitely many $t$ tending to infinity, $P_X(s)$ can be quite large, namely    
\begin{equation}\label{P is big}
|P_X(\s + i t) | \gg \exp{(X^{1-2\s}/\sqrt{\log X} ) }\,.
\end{equation}

In this section we have seen  that short truncations of its Euler product 
approximate $\zeta(s)$ well  in the region $\s>1$, $|s-1|>\frac{1}{10}$. 
We also showed that this remains true in the right-half of the critical strip
if the Riemann Hypothesis is true and if we are not too near the critical line 
(and use a weighted Euler product). 
However,  to the left of the critical line  the Euler product is not a good
approximation of $\zeta(s)$ regardless of  how long  it is.


\section{Products, sums, and moments} 

Our purpose in this section is to deduce two consequences of 
the results of the previous  section.  First we require a result whose
proof is given in the Appendix.   
\begin{thm}\label{lem on RH 1}
Assume the  Riemann Hypothesis.
Let $\s\geq \frac12$ be bounded, 
 $|s-1|> \frac{1}{10}$,  and $2\leq X \leq t^2$.
Then there is a positive constant $C_1$ such that
\begin{equation}\label{Riemann approx 4}
 \zeta(s) = \sum_{n \leq X} n^{-s}   
+  O\left( X^{\frac12 - \sigma } e^{C_1 \Phi(t)} \right)   \,.
 \end{equation}
\end{thm}

If  $\frac12 + 2C_1\frac{\Phi(t)}{\log X}\leq \s \leq 1$, 
the  error term  in Lemma~\ref{lem on RH 1} is $O(e^{-C_1\Phi(t)})$. 
For the same $\s$-range, the approximation of $\zeta(s)$ given by 
Theorem~\ref{thm on RH}   is
\begin{equation}\label{zeta aprox }
\zeta(s) = P_X(s) \left(    1 + O \left(e^{-(2-\e)C_1\Phi(t)}  \right)  \right)  ,
\end{equation}
where $\e$ is arbitrarily  small.
Thus, equating respective sides of \eqref{Riemann approx 4} and
\eqref{zeta aprox } and solving for $P_X(s),$ we see that
\begin{align*}\label{ }
 P_X(s) =   \left( \sum_{n \leq X} \frac{1}{n^s}   
+  O\left( e^{-2C_1 \Phi(t)} \right) \right)
   \bigg(    1 + O \left(e^{-(2-\e)C_1\Phi(t)}  \right)  \bigg) \,.
\end{align*}
By the corollary to Theorem~\ref{lem on RH 1}  (see the Appendix), the sum 
here is $\ll e^{C_1\Phi(t)}$, so we obtain
\begin{align*} 
 P_X(s) =     \sum_{n \leq X} \frac{1}{n^s}   
+  O\left( e^{-(1-\e)C_1 \Phi(t)} \right)  \,.
\end{align*}
We have now proved

\begin{thm}\label{P=Sum}  
Assume the  Riemann Hypothesis.
Let $\s\geq \frac12$ be bounded, 
 $|s-1|> \frac{1}{10}$,  and $2\leq X \leq t^2$.
There is a positive constant $C_1$ such that if
$  \frac12 + \frac{2C_1\Phi(t)}{\log X } \leq \s \leq 1$, then
 \begin{equation*}\label{approx of P and Sum}
\sum_{n \leq X} \frac{1}{n^s}  
 =  P_X(s)   +  O \left(e^{-C_2\Phi(t)}  \right)  \,,
\end{equation*}
for any positive constant $C_2$ less than $ C_1$.
\end{thm}

Our second   observation is that one can use these appoximations
to calculate the moments of a very long Euler product.
Suppose one wished to compute  the moments
$$
\int_{0}^{T} |P_{X}(\sigma+i t)|^{2k} \;d t\,,
$$
when $\frac12<\s<1$. The standard method would be to write 
$P_{X}(s)^{k}$ as a Dirichlet series and use a mean value theorem
for such polynomials to compute the mean modulus squared. But this 
only works well when the product does not have  many factors. For example, 
for a  slightly different Euler  product, Gonek, Hughes and Keating~\cite{GHK} 
have proved the  
 
\begin{thm*}\label{thm:Q moments}
Let  $1/2 < c <1$, $\epsilon>0$,
and let $k$ be any positive real number.
Suppose that $X$ and $T \to \infty$ and $X = O\left((\log
T)^{1/(1-c+\epsilon)}\right)$. Then we have
\begin{align*}
 \int_T^{2T}  \left| \exp \left(\sum_{n\leq X }  
 \frac{\Lambda(n)}{ n^{\s+it}  \log n }
  \right) \right|^{2k} \;  d t  
  \sim     a_k(\s)\, T 
\zeta(2\sigma)^{k^2} e^{-k^2 E_1\left(  (2\sigma-1)\log X  \right)} \,,
\end{align*}
where
$$
a_k(\s) =     \prod_{p } \left\{\left(1-\frac{1}{p^{2\sigma}}\right)^{k^2}
    \sum_{m=0}^{\infty}\frac{d_k(p^m)^{2}}{p^{2m\sigma}} \right\}
$$
uniformly for $c \leq \sigma\leq 1$.
Here
  $d_k(n)$ is the $k$th divisor function and 
  $E_1(z)=\int_{z}^{\infty}\frac{e^{-w}}{w} dw$ is the first exponential integral.
 \end{thm*}

Note that here the number of factors  in  the Euler product is not even $\log^2T$.   
On the other hand, if we assume the Riemann Hypothesis, 
that  $ \frac12 +  \frac{ C\log\log 2T }{\log X} \leq \s <1$ with $C>1$,
and that $ 2 \leq X \leq T^2$, then by Theorem~\ref{Approx by P}  
\begin{equation*}
\zeta(s) = P_X(s)\big(1 + o(1)\big)\,.
\end{equation*}
Hence,
\begin{equation*}
\int_{T}^{2T} |P_{X}(\sigma+i t)|^{2k} \;d t \sim
\int_{T}^{2T} |\zeta(\sigma+i t)|^{2k} \;d t \,.
\end{equation*}
Now, it is a consequence of the
Lindel\"{o}f Hypothesis (Titchmarsh~\cite{T}, Theorem  13.2), 
and so also of the Riemann Hypothesis,
that when $\frac12<\s<1$ is fixed,   
$$
\int_{T}^{2T} |\zeta(\sigma+i t)|^{2k} \;d t 
\sim T \sum_{n=1}^{\infty} \frac{d_{k}^{2}(n)}{n^{2\s}}  
$$
for any fixed positive integer $k$.  
Thus, for such $\s$ and $k$ we have
\begin{equation*}
\int_{T}^{2T} |P_{X}(\sigma+i t)|^{2k} \;d t 
\sim T \sum_{n=1}^{\infty} \frac{d_{k}^{2}(n)}{n^{2\s}} \,.
\end{equation*}
This gives an estimation of  the moments of  an extremely long Euler product
deep into the critical strip.


\section{A function related to the zeta-function }

In Section~\ref{series approx} we showed that short truncations of its  Dirichlet 
series approximate  $\zeta(s)$  in $\s>1$ and that, 
if the Lindel\"{o}f Hypothesis is true,  this also holds in $\s> \frac12 $. 
The approximation cannot be good  in the strip $0<\s \leq  \frac12$ unless 
the length of the sum is of order at least $t $; and this is  so
even if we assume the Lindel\"{o}f or  Riemann Hypothesis.
In Section~\ref{product approx}  we showed that the situation is 
similar, up to a point, when we approximate  $\zeta(s)$ by  the 
weighted Euler product $P_X(s)$: short   products  approximate  $\zeta(s)$ 
well in the half-plane $\s>1$ unconditionally, and in the strip 
$\frac12< \s \leq 1$ on the Riemann Hypothesis. However, the approximation 
 cannot be close in $0<\s <  \frac12$   no matter how 
many factors there are, for   $P_X(s)$  gets much 
larger than $\zeta(s)$ in this strip (see \eqref{P is big}).   

We   now reexamine the approximation of $\zeta(s)$ 
by   sums 
when $\s$ is close to  $\frac12$. If we assume  the Riemann Hypothesis,
then by \eqref{Riemann approx 4}
\begin{equation*}\label{Riemann approx 6}
 \zeta(s) = \sum_{n \leq X} \frac{1}{n^s}   
+  O\left( e^{-C_1\Phi(t)} \right) 
\end{equation*}
for $ \frac12 +  \frac{2C_1\Phi(t)}{\log X} \leq\s \leq 1$  and
 $2\leq X \leq t^2$. This is good for $X $  a small power of $t$ 
as long as $\s$ is not too close to $\frac12$, but
we know   $X$ has to be of order $t$ on $\s=\frac12$.   
This means the approximation   is  off by about
$ \sum_{X< n \leq t} n^{-s} $.
The Hardy-Littlewood approximate functional equation~\cite{HardyLittle} 
(or see Titchmarsh \cite{T}),
gives us  another way to express this. It says that 
\begin{equation}\label{strong afe 1}
\zeta(s) = \sum_{n\leq X}\frac{1}{n^s} + 
\chi(s)\sum_{n\leq |t|/2\pi X}\frac{1}{n^{1-s}}
+ O( X^{-\sigma} )  + O( \t^{-\frac12} X^{1-\sigma} )  \,,
\end{equation}
where  $0 \leq \sigma  \leq 1$,  $|s-1| \geq\frac{1}{10}$,  
and $\chi(s)$ is the  factor 
in the functional equation  
\begin{equation*}\label{f. e.}
\zeta(s) = \chi(s) \zeta(1-s) \,.
\end{equation*}
From this  we see that the  amount by which the  sum 
$\sum_{n\leq X} n^{-s}$ is off 
from $\zeta(s)$  is  about
$$
 \chi(s)\sum_{n\leq |t|/2\pi X}\frac{1}{n^{1-s}}\,.
$$
If we Let $X= \sqrt{|t|/2\pi }$ and note that $|\chi(\frac12 +it)|=1$, 
we find  that on the critical line $\zeta(s)$ is essentially  
composed of two pieces of equal size:
\begin{equation*}\label{strong afe 2}
\zeta(\tfrac12 +it) = \sum_{n\leq X}\frac{1}{n^{\frac12 +it}} + 
\chi(\tfrac12 +it)\sum_{n\leq  X}\frac{1}{n^{\frac12 -it}}
+ O( \t^{-\frac14} )    \,.
\end{equation*}

In the case of  Euler products,   
Theorem~\ref{Approx by P} suggests that  $P_X(s)$ 
approximates  $\zeta(s)$ well even closer to 
the critical line than a  sum  of length $X$ does. 
For $P_X(s)$ is a good approximation  
when  $\s \geq \frac12 + \frac{C\log\log 2\t}{log X}$, 
while the sum  is only close
(as far as we know)  when
$\s \geq \frac12 + \frac{2C_1\Phi(t)}{log X}$. 
In light of this and \eqref{strong afe 1} 
it is tempting to guess that  
 \begin{equation}\label{wrong afe}
\zeta(s) \approx P_X(s)  + \chi(s) P_X(1-s)  \,,
\end{equation}
for some unspecified $X$.
However, this is not a good guess.
For we have seen that $P_X(1-s)$ gets as large as 
$\exp{(X^{\s-\frac12}/\log X)}$ when $\s >\frac12$, whereas 
$\sum_{n\leq X} n^{s-1}$ is no larger than   
$X^{\s-\frac12 }e^{C_1\Phi(t)}$ by Corollary~\ref{bound for sums}. 

The difficulty here is crossing the line $\s =\frac12$, where there is a 
qualitative change in the behavior of the zeta-function. A way around this is to 
use the fact that the functional equation tells us the zeta-function everywhere 
once we know it in  $\s \geq \frac12$. If we  restrict our attention to  this 
half-plane,  a reasonable alternative to \eqref{wrong afe} is
 \begin{equation*}\label{right afe}
 \zeta_X(s)=  P_X(s)  + \chi(s) P_X(\overline{s})  \,.
\end{equation*}
Note  that on the critical line,
$ \zeta_X(s)$ and the right-hand side of \eqref{wrong afe} are identical. 
Also, since   $\chi(s) \ll t^{\frac12-\s}$  for  
$\s > \frac12$  (see \eqref{chi 2} below), we have   
$$
 \zeta_X(s)= P_X(s) (1+ O(t^{\frac12-\s})) \,.
$$
Combining this observation with Theorem~\ref{Approx by P2}, we obtain
\begin{thm}\label{Approx of zeta by zeta_X} 
Assume the Riemann Hypothesis.  Let  
 $  2 \leq X \leq t^2$,  $|s-1|\geq \frac{1}{10}$,
and  $ \frac12 +   \frac{ C \log\log 2\t }{\log X} \leq \s \leq 1$ with $C >1$.
Then  
\begin{equation}\label{approx by P4}
\zeta(s) =  \zeta_X(s) \left(    1 + O \left( \log^{(1-C)/2} \t \right)  \right) \,. 
\end{equation} 
Conversely, if \eqref{approx by P4}   holds  for $2 \leq X \leq t^2$
in the region stated, then $\zeta(s)$ has at most a 
finite number of zeros to the right of  
$\s= \frac12 +   \frac{ C \log\log 2\t }{\log X}$.
\end{thm}

Thus, even though  $ \zeta_X(s)$ is not analytic,   
it   approximates $ \zeta(s)$ well to the right of the critical line. 
It resembles the zeta-function closely in other ways too as we shall see.

\section{The Riemann Hypothesis for $\zeta_X(s)$}

For a closer study of  $ \zeta_X(s)$  we require several  properties of 
  the chi-function  
  \begin{equation*}\label{chi 1}
  \chi(s)  =  \pi^{s-\frac12}
  \frac{ \Gamma(\frac12-\frac12 s)  }{ \Gamma( \frac12 s) } \,,
\end{equation*} 
which  appears in the functional equation of the zeta-function.
Chi has simple poles  at $s= 1, 3, 5, \ldots$ from the 
$\Gamma$-factor in the numerator. If we stay away 
from these,
\begin{equation}\label{chi 2}  	
   \chi(s)  =  \left( \frac{ \t }{2\pi} \right)^{\frac12 - \sigma - \i t} 
   e^{\i t   +\frac14 \i \pi}	\left\{ 1 + O\left(\frac{1}{\t}\right)  \right\}  
\end{equation} 
in any  half-strip $ -k <\sigma < k$,  $t \geq 0$, by Stirling's approximation.
When $ t < 0,  \chi(s) $ is given by the conjugate of this. We
note for later use that the $O$-term is differentiable. 

Clearly $|\chi(\frac12+it)| =1$ for all $t$. The converse   is  also  almost  true. 
\begin{lem}\label{chi lemma} 
There is a positive absolute constant $C_0$ such that
if $ |\chi(\s+it)|=1$ with $0 \leq \s \leq 1$ and $|t|\geq C_0$, then $\s=\frac12 $.
\end{lem}
\begin{rem*} One can take $C_0 <6.3$, but we do not require this. 
\end{rem*}
\begin{proof}  
Taking the logarithmic derivative of
 \eqref{chi 2} by means of Cauchy's integral formula, 
 we find that 
 \begin{equation*}
\Re\, \frac{\chi^{'}}{\chi}(s) = -\log \frac{\t}{2\pi} + O(\frac{1}{\t}) \,.
\end{equation*}
Since $|\chi(\frac12+it)| =1$, we see that if $\s_1 > \frac12$, then
\begin{align*}
\log |\chi(\s_1+it)| =& \int_{\frac12}^{\s_1} \Re \, \frac{\chi^{'}}{\chi}(s)\,d\s \\
=&  \left(\frac12-\s_1\right)\log \frac{\t}{2\pi} + O\left(\frac{\s_1-\frac12}{\t}\right) \,.
\end{align*}
This is negative for all $t$ sufficiently large  (independently of $\s_1 $), so the result follows.
The proof is similar for $\s_1 < \frac12$.  
\end{proof}

From now on $C_0$ will denote  the constant in Lemma~\ref{chi lemma}.

We now prove

\begin{thm} {\bf (The Riemann Hypothesis for {\boldmath $\zeta_X(s)$})} 
Let $\rho_X = \beta_X + i \gamma_X$ 
denote any zero of $\zeta_X(s)$
with  $0\leq \beta_X \leq 1$ and $\gamma_X \geq C_0$,  the constant in Lemma~\ref{chi lemma}.
Then   $ \beta_X  =  \frac12$. 
\end{thm}
\begin{proof}Since $ \zeta_X(s) = P_X(s) + \chi(s) P_X(\overline{s})$ and 
$P_X(s)$  never vanishes,  the  zeros of $\zeta_X(s)$ can only occur at points 
where $|P_X(s)| = |\chi(s ) P_X(\overline{s})|$, that is, where$ |\chi(s )|=1$.
The result now follows from Lemma~\ref{chi lemma}.
\end{proof}


\section{The number of zeros of $\zeta_X(s)$}

In this section we  estimate the number of zeros of  $\zeta_X(s)$ 
up to height $t$ on the critical line and show, among other things,    
that it has  at least as many zeros (essentially) as $\zeta(s)$
does, namely 
\begin{equation}\label{N(t)}
N(t) =  \frac{t}{2\pi}\log \frac{t}{2\pi} - \frac{t}{2\pi} 
+ \frac78 +S(t)+ O\left(\frac{1}{\t}\right)  \,.
\end{equation}
An exact expression  is 
 \begin{equation}\label{N(t) 2}
N(t) =  -\frac{1}{2\pi}\arg \chi(\tfrac12+it)  + S(t)  + 1 \,,
\end{equation}
and we shall use this later. 
Here the argument of $\chi$   is determined by
starting with the value $0$ at $s=2$ and letting it  
vary continuously, first along the segment from $2$  to $2+it$, 
and then horizontally from $2+it$ to $\frac12+it$.

To investigate the zeros of $\zeta_X(s)$ we write
\begin{equation}\label{F formula 1}  
 \zeta_X(s) = P_X(s)\left( 1+ \chi(s) \frac{P_X(\overline{s})}{P_X(s)} \right) .
\end{equation} 
Since  $P_X(s)$ is never zero, $\zeta_X(s)$   vanishes if and only if 
$\chi(s) P_X(\overline{s})/P_X(s)= -1$. Now $|P_X(\overline{s})/P_X(s)| \\=1$,
so  this is equivalent to  $|\chi(s)|=1$ and 
$\arg\left(\chi( s) P_X(\overline{s})/P_X( s)\right)  \equiv \pi \pmod{2\pi}$.
By Lemma~\ref{chi lemma}, if $|\chi(s)|=1$ in the 
half-strip $0\leq \s \leq 1$, $t \geq C_0$, then   $\s=\frac12$.  
Conversely, we know that $|\chi(\frac12+it)|=1$ for all $t$. Thus,
$\zeta_X(s)=0$ in $0\leq \s \leq 1$, $t \geq C_0$ if and only if 
$\arg\left(\chi( \frac12+ i t) P_X( \frac12 - i t)/P_X( \frac12+ i t)\right) 
\equiv \pi \pmod{2\pi}$.
Defining  
\begin{equation}\label{arg formula 1}
F_X(t)= -\arg\chi(\tfrac12+ i t) + 2 \arg P_X( \tfrac12+ i t), 
\end{equation}
we see that when $t \geq C_0$,  $\zeta_X( \frac12+ i t) =0$ if  
and only if
\begin{equation*}\label{arg formula 2}
F_X(t)  \equiv \pi (mod\, 2\pi)\;.
\end{equation*}
This  will be the basis for much of our further work.

Before  turning to our first estimate,
we point out  that, as with $\arg \chi(\frac12+it)$,  $\arg P_X(\frac12+it)$
is defined by continuous variation along the  segments $[2, 2+it]$ and $[2+it, \frac12+it]$,
starting with the value $0$ at $s=2$. Also note from \eqref{arg formula 1} 
that $F_X(t)$ is infinitely differentiable for $t>0$.

We now prove
 
\begin{thm}\label{lower bd on zeros}
Let $N_X(t)$ denote the number of zeros $ \rho = \frac12+\i \gamma_X$
of $\zeta_X(s)$  with $0\leq \gamma_X \leq t$. Then
\begin{equation*}\label{formula for zeros}
N_X(t) \geq \frac{t}{2\pi}\log \frac{t}{2\pi} - \frac{t}{2\pi} 
- \frac{1}{\pi}\sum_{n\leq X^2}  
\frac{\Lambda_X(n)\sin(t \log n)}{n^{\frac12}\log n}
+O_X(1) \,.
\end{equation*}        
\end{thm}

\begin{proof} 
There are at most  finitely many zeros (the number may depend on $X$)
 with ordinates between
$0$ and $C_0$  (the constant in Lemma~\ref{chi lemma}).  
We may therefore assume  that $t \geq C_0$.       
Now, by \eqref{chi 2}  
\begin{equation*}\label{chi 4}  	
\arg   \chi(\frac12+\i t)  =  - t \log \frac{ t }{2\pi} + t + \frac14 \pi
+ O\left(\frac{1}{\t}\right),
\end{equation*} 
and by \eqref{P}
\begin{align}
\arg P_X( \tfrac12+i t)  
 = \Im  \sum_{n\leq X^2}\frac{\Lambda_X(n)}
 {n^{\frac12+i t}\log n}  
 = - \sum_{n\leq X^2}\frac{\Lambda_X(n) \sin(t\log n)}
 {n^{\frac12}\log n} \,.\notag
\end{align}
Thus, we may  express $F_X(t)$ in \eqref{arg formula 1} as
\begin{equation}\label{arg formula 4}
F_X(t)= t \log \frac{ t }{2\pi} - t -\frac14 \pi - 2 \sum_{n\leq X^2}\frac{\Lambda_X(n)\sin(t\log n)}
 {n^{\frac12}\log n}
+ O\left(\frac{1}{\t}\right) \,.
\end{equation}
Recall that  $\zeta_X(\frac12+\i t)= 0$ when $t \geq C_0$ if and only if 
$F_X(t)  \equiv \pi (mod\, 2\pi)$.
Since  $F_X(t)$  is continuous,  this happens at least
\begin{equation*}
\frac{t}{2\pi} \log \frac{ t }{2\pi} - \frac{t}{2\pi} 
- \frac{1}{\pi} \sum_{n\leq X^2}\frac{\Lambda_X(n)\sin(t\log n)}
 {n^{\frac12}\log n} + O_X(1) 
\end{equation*}
times as on $[C_0, t]$. 
This gives the result.
\end{proof}

The sum over prime powers in \eqref{arg formula 4}  obviously plays an 
important role in producing zeros of $ \zeta_{X}(s)$.  This sum
is just $ -\arg P_X( \frac12+i t)$, but it will be convenient to give it a 
simpler name.  Thus, from now on we 
write
\begin{equation*}\label{f_X} 
f_X(t) =  -\arg P_X( \tfrac12+i t) 
=  \sum_{n\leq X^2}  
\frac{\Lambda_X(n)\sin(t \log n)}{n^{\frac12}\log n}\,.
\end{equation*}

As is well known (see Selberg~\cite{S2} or Titchmarsh~\cite{T}), 
$-(1/\pi)f_X(t)$ is a good approximation in mean-square to 
$S(t)=  (1/\pi)\arg\zeta(\frac12+it)$ if $X$ is a small power of $t$
and if the Riemann Hypothesis holds. 
However, a  closer analogue of 
$S(t)$ is   
\begin{equation*}
S_X(t) = \frac{1}{\pi}\arg \zeta_X(\tfrac12+it) .   
\end{equation*}
From \eqref{F formula 1} we see that
\begin{align*}\label{S_X formula}
S_X(t) =& -\frac{1}{\pi} \sum_{n\leq X^2}  
\frac{\Lambda_X(n)\sin(t \log n)}{n^{\frac12}\log n}
- \frac{1}{\pi} \arg \left(  1+ \chi(\tfrac12+it)       \notag
\frac{P_X(\tfrac12- it)}{P_X(\tfrac12+it)} \right)   \\
=& -\frac{1}{\pi} f_X(t)
- \frac{1}{\pi} \arg \left(  1+ e^{- i F_X(t)} \right).
\end{align*}
The second term on the right, which  contains the jump discontinuities of $S_X(t)$ 
as $t$ passes through zeros of  $\zeta_X(\frac12+it)$,  has modulus $\leq \frac12$. 
(Note that our convention is that  the argument is $\pi/2$ when $1+ e^{- i F_X(t)}$ vanishes).  Thus, 
$S_X(t)$  and $-(1/\pi)f_X(t)$ differ by at most $O(1)$.  

The next theorem shows that  when $X$ is not too small,  $f_X(t)$ and $ S_X(t)$ have 
the same bound as  $S(t)$, namely  $\Phi(t)$.

 \begin{thm}\label{Bound on arg sum} 
Assume the Riemann Hypothesis and that $2  \leq X \leq t^2$. Then
\begin{equation*}
f_X(t)=  \sum_{n\leq X^2}  
\frac{\Lambda_X(n)\sin(t \log n)}{n^{\frac12}\log n}  \ll   \Phi(t) + \frac{\log \t }{\log X}.
\end{equation*}
In particular,  $f_X(t) \ll \Phi(t)$ when $\exp(\log \t/\Phi(t) ) \leq X \leq t^2$.  
The same bounds hold for $ S_X(t)$. 
 \end{thm}
 
\begin{proof} Since $ S_X(t)+ (1/\pi)f_X(t) \ll 1$, 
it suffices to prove the result for $f_X(t)$.
By \eqref{R_X} and \eqref{arg zeta} we have
\begin{equation*}
S( t) = - \frac{1}{\pi} \sum_{n\leq X^2} 
  \frac{\Lambda_X(n) \sin(t \log n)}{n^{\frac12} \log n} 
+ O\left(\Phi(t)\right) 
+ O\left(\frac{\log \t}{\log X}\right) \,.
 \end{equation*}
Since $S(t) \ll \Phi(t)$, the result follows. 
\end{proof}

From Theorem~\ref{Bound on arg sum}  and Theorem~\ref{lower bd on zeros}
we immediately obtain
\begin{thm}
Assume the Riemann Hypothesis is true. Then for   $2 \leq X \leq t^2$,   
 \begin{equation}\label{lower bd}
N_X(t) \geq \frac{t}{2\pi}\log \frac{t}{2\pi} - \frac{t}{2\pi} + O( \log \t )  \,.
\end{equation}
Moreover, if   $\exp(c\log \t/\Phi(t) ) \leq X \leq t^2$,  
 where $c$ is any positive constant, then
\begin{equation*}
N_X(t) \geq \frac{t}{2\pi}\log \frac{t}{2\pi} - \frac{t}{2\pi} 
+O(\Phi(t)) \,.
\end{equation*}
\end{thm}

To obtain an upper bound for $N_X(t)$ of the same order we require the
 following theorem.


\begin{thm}\label{Bound on real part sum} 
Assume the Riemann Hypothesis and that $2  \leq X \leq t^2$. Then
\begin{equation*}\label{bd im part sum} 
 f_{X}^{'}(t) =  \sum_{n\leq X^2} 
  \frac{\Lambda_X(n) \cos(t \log n)}{n^{\frac12}  } 
 \ll   \Phi(\t)\log X + \frac{\log \t}{\log X} \,.
 \end{equation*}
\end{thm}

\begin{proof} Taking the real part of  \eqref{formula 1} 
with $\s= \frac12$, we obtain
\begin{align}\label{re log deriv 1}
\Re\; \frac{\zeta^{'}}{\zeta}(s) 
=  &     \sum_{\g} \frac{\cos((\g-t)\log X )  - \cos(2(\g-t)\log X ) }{( \g-t)^2  \log X }  
 -  \sum_{n\leq X^2} \frac{\Lambda_X(n) \cos(t \log n)}{\sqrt{n}  } \\
 &+ O\left( \frac{ 1}{\log X} \right) \,. \notag
\end{align}
Similarly, from  
\begin{equation*}
\frac{\zeta^{'}}{\zeta}(s) = \sum_{\rho} 
\left(\frac{1 }{s-\rho   } + \frac{1}{\rho} \right)
-\frac12 \log \frac{\t}{2\pi} +O(1) 
\end{equation*}
we see that
\begin{equation*}
\Re \frac{\zeta^{'}}{\zeta}(s) =  
-\frac12 \log \frac{\t}{2\pi} +O(1)\,. 
\end{equation*}
We substitute this into the left-hand side of \eqref{re log deriv 1}
and rearrange and  find that
\begin{align*} 
\sum_{n\leq X^2}   \frac{\Lambda_X(n) \cos(t \log n)}{\sqrt{n} } 
 \; =\tfrac12 \log \frac{\t}{2\pi}   
+    \sum_{\g} \frac{\cos((\g-t)\log X )  - \cos(2(\g-t)\log X ) }{( \g-t)^2   \log X } 
+ O(1)\,.
\end{align*}
In the sum over zeros, the terms  with   $|t-\g| \geq 1$ contribute
$ 
\ll \log \t /\log  X \,.
$ 
Thus, writing $C(v) = \cos(v \log X) -\cos(2v \log X)$, we have

\begin{equation}\label{step 1}
 \sum_{n\leq X^2}  \frac{\Lambda_X(n) \cos(t \log n)}{\sqrt{n} } 
  =\frac12 \log \frac{\t}{2\pi}   
+  \frac{1}{\log X}  \sum_{|\g-t| \leq 1} \frac{C(\g-t) }{( \g-t)^2  } 
+ O\left(\frac{\log \t}{\log X}\right)\,. 
\end{equation}

To estimate the sum on the last line, first note that

\begin{equation}\label{C est}
C(v)  =
\begin{cases}
\frac32 v^2 \log^2 X  +O( |v|^4 \log^4 X) &\quad \hbox{if}\quad  |v| \leq 1/\log X \,, \\
O(1)  &\quad \hbox{if} \quad |v| > 1/\log X \,,
\end{cases}
\end{equation}
and that
\begin{equation*}\label{C' est}
C'(v) = 
\begin{cases}
3 v  \log^2 X  +O( |v|^3 \log^4 X) &\quad \hbox{if}\quad  |v| \leq 1/\log X \,, \\
O(\log X)  &\quad \hbox{if} \quad |v| > 1/\log X \,.
\end{cases}
\end{equation*}
In particular, it follows that
\begin{equation}\label{C/v^2}
\frac{d}{dv}\left(\frac{C(v) }{v^2  }\right) 
 = \frac{v^2 C^{'}(v) -2 v C(v)}{v^4}
 \ll 
\begin{cases}
 |v| \log^4 X &\quad \hbox{if}\quad  |v| \leq 1/\log X \,, \\
v^{-2} \log X   &\quad \hbox{if} \quad |v| > 1/\log X \,.
\end{cases} 
\end{equation}

Now, by \eqref{N(t)}  
\begin{align*}
  \sum_{|\g-t| \leq 1} \frac{C(\g-t) }{( \g-t)^2  } 
&=  \int_{|u-t | \leq 1} \frac{C(u-t) }{( u-t)^2  } \;dN(u) \\
& =   \frac{1}{2\pi} \int_{|v | \leq 1} \frac{C(v) }{v^2  } 
  \log \left(\frac{t+v}{2\pi}\right) \;dv \; +  \int_{|v | \leq 1} \frac{C(v) }{v^2  } 
\;dS(t+v) \;. \notag
\end{align*}
Using \eqref{C/v^2}, we see that the integral with respect to $dS $ is, 
\begin{align*}
  S(t+v)\;&\frac{C(v) }{v^2  }\bigg|_{ -1}^{ 1} 
 -  \int_{|v | \leq 1} S(t+v)
  \frac{d}{dv}\left(\frac{C(v) }{v^2  }\right) \;dv       \notag   \\
  & \ll    \Phi(\t)    + 
  \Phi(\t) \left( \int_{|v | \leq 1/\log X} |v|  \log^4X \;dv  
  +     \int_{1/\log X < |v | \leq 1} \frac{ \log X}{v^2} \;dv  \right)  \\
 & \ll  \Phi(\t) \; \log^2 X 
\,. \notag
\end{align*}
The other integral is  
\begin{align}\label{step 2}
 \frac{1}{2\pi} \int_{|v | \leq 1}  
\left( \log  \frac{t}{2\pi} + O\left(\frac{|v|}{t}  \right)   \right) 
 \frac{C(v) }{v^2  }  \;dv \,.\notag
\end{align}
The $O$-term contributes
\begin{equation*}\label{O-term}
\ll  \frac{1}{\t } \left( \log^2  X
 \int_{|v | \leq 1/\log X}  |v| \;dv    +  \int_{ 1/\log X <|v | \leq 1}  |v|^{-1} \;dv  \right)
 \ll   \frac{1}{\t } \, \log \log X  
\end{equation*}
by \eqref{C est}.
Thus, combining these results, we find that
\begin{equation}\label{sum g}
  \sum_{|\g-t| \leq 1} \frac{C(\g-t) }{( \g-t)^2  } 
= \frac{1}{2\pi  } \log \frac{t}{2\pi}  \;\int_{|v | \leq 1}  \frac{C(v) }{v^2  } 
  \;dv  \;+ \;O(\Phi(\t) \log^2 X) \,.
\end{equation}

To calculate  the integral we write
\begin{align*}\label{to do}
  \int_{|v | \leq 1}  \frac{C(v) }{v^2  }  \;dv 
  =   \int_{-\infty}^{\infty}  \frac{C(v) }{v^2  } \;dv 
-   \int_{|v | >1}  \frac{C(v) }{v^2  } 
  \;dv \;.
\end{align*}
By \eqref{C est} the second integral is  $O(1)$.
By the calculus of residues and the definition of 
$C(v)$, the first   equals

\begin{equation*}
\begin{aligned} \label{main term}
 \int_{-\infty}^{\infty}  \frac{\cos(v \log X) -\cos(2v \log X)}{v^2  } \;dv 
 &=  \Re \;  \int_{-\infty}^{\infty}  \frac{e^{i v \log X} - e^{2 i v \log X} }{v^2  } \;dv   \notag \\
 & = \Re \; 2\pi i   \left(-\tfrac12 \;\mbox{Res}_{v=0} \; \frac{e^{i v \log X} - e^{2 i v \log X} }{v^2  } \right) \\
 &= - \Re \;\pi i\; \left( -i \log X  \right) = - \pi \log X \,.    
 \end{aligned}
 \end{equation*}
 Thus, 
 $$
 \int_{|v | \leq 1 }  \frac{C(v) }{v^2  }  \;dv = - \pi \log X + O(1)\,.
 $$
 
Using this in \eqref{sum g}, we obtain
 \begin{equation*}
  \sum_{|\g-t| \leq 1} \frac{C(\g-t) }{( \g-t)^2  } 
= -\frac{1}{2  }\; \log \frac{t}{2\pi}\; \log X  \;+ \;O(\Phi(\t) \log^2 X) 
 \;+ \;O(\log \t) \,. 
\end{equation*}
It therefore  follows from \eqref{step 1} that
\begin{align*}\label{}
\sum_{n\leq X^2}   &\frac{\Lambda_X(n) \cos(t \log n)}{\sqrt{n} } 
  \ll  \Phi(\t) \log X  \;+\;  \frac{\log \t}{\log X} \,.
\end{align*}
This completes the proof of the theorem.
\end{proof}

The zeros of $\zeta_X(\s+it)$ with $t\geq C_0$
 arise as the solutions of   
 \begin{equation*}\label{arg formula 3}
F_X(t)  \equiv \pi (mod\, 2\pi),
\end{equation*}
and their number  in $[C_0, t]$ is at least
$(1/2\pi)\,F_X(t) +O_X(1)$ because this is the minimum number of times
the curve $y= F_X(t)$ crosses the horizontal lines
$y= \pi, 3\pi, 5\pi, ...$.
However, there could be ``extra'' solutions if  $F_X(t)$ is not monotone 
increasing.  Now 
\begin{equation}\label{F'}
F_{X}^{'}(t) = \log \frac{t}{2\pi} - 2\sum_{n\leq X^2} 
  \frac{\Lambda_X(n) \cos(t \log n)}{n^{\frac12}  } 
  + O\left(\frac{1}{\t}\right)\,.
\end{equation}
By Theorem~\ref{Bound on real part sum} there exists a positive 
constant $C_3$, say, such that if $X \leq \exp(C_3\log \t/\Phi(t))$ 
and $t$ is large enough, then
\begin{equation}\label{pos deriv}
 \sum_{n\leq X^2} 
  \frac{\Lambda_X(n) \cos(t \log n)}{n^{\frac12}  }  
  <  \frac12 \log\frac{t}{2\pi} \,.
\end{equation}
This means $F_X^{'}(t)$ is positive, so 
$F_X(t)  \equiv \pi (mod\, 2\pi)$ has no  extra  solutions.
We have therefore proved

\begin{thm}\label{exact zero count} Assume  the Riemann Hypothesis. There is a constant $C_3>0$ 
 such that if  $X <  \exp\big(C_3 \log t /  \Phi(t) \big)$, then 
\begin{equation}\label{N formula 1}
 N_X(t) =
\frac{t}{2\pi} \log \frac{t}{2\pi} -\frac{t}{2\pi}  
 - \frac{1}{\pi}\sum_{n\leq X^2}\frac{\Lambda_X(n) \sin(t\log n)}
 {n^{\frac12}\log n} +O_X(1) \,.
\end{equation}
Less precisely,  
\begin{equation*}\label{N formula 2}
 N_X(t) =
\frac{t}{2\pi} \log \frac{t}{2\pi} -\frac{t}{2\pi}  
   +O_X(\Phi(\t))\,.
\end{equation*} 
\end{thm}

It would be interesting and useful to know whether \eqref{N formula 1} 
(perhaps with a  larger $O$-term) also holds when $X$ is a small fixed  power of $t$.
If that is the case,  classical results about the statistics 
of the  zeros of $\zeta(s)$ whose proofs depend  on approximating $S(t)$
by the trigonometric polynomial $-(1/\pi) f_X(t)$ would   hold for the zeros of $\zeta_X(s)$ as well. 
What we can show for larger $X$  is  the following unconditional result.

\begin{thm}\label{no. of zeros}   
There exists a positive constant $C_4$ such that if $X \leq t^{C_4}$, then
$$
N_X(t) \ll t\log t \,.  
$$
Moreover, if \,$\log X/\log t =o(1)$, then
\begin{equation*}\label{N asymp}
 N_X(t) = \left(1+ o(1) \right)
\frac{t}{2\pi} \log \frac{t}{2\pi} \,.
\end{equation*}
\end{thm}

\begin{proof}
There are two ways  solutions to $F_X(t)  \equiv \pi (mod\, 2\pi)$
may arise, and we  shall refer to the zeros of $\zeta_X(s)$ corresponding 
to these two ways as zeros of the ``first'' and ``second'' kind. 

The first way  is  by $F_X(t)$ increasing or decreasing from one   
odd multiple of $\pi$, say $(2k+1)\pi $, 
to  the next larger or smaller odd multiple
of $\pi$, without first crossing $(2k+1)\pi $
again.  A moment's reflection reveals that 
the total number of distinct zeros  in $[C_0,\, t]$ arising this way is 
big-$O$  of  the total variation of
$ F_X(t) $, namely,
\begin{equation*}
\frac{1}{2\pi} \int_{C_0}^{t}  |F_{X}^{'}(u)| \;du \;.
\end{equation*}
By \eqref{F'} and the triangle and Cauchy-Schwarz inequalities, this is 
\begin{align*} 
 &\leq \frac{1}{2\pi} \int_{C_0}^{t}  
 \left| \log \frac{u}{2\pi} - 2\sum_{n\leq X^2} 
  \frac{\Lambda_X(n) \cos(u \log n)}{\sqrt{n}  }  \right| \;du
 +O(1) \\
& \leq  \frac{t}{2\pi} \log \frac{t}{2\pi} -  \frac{t}{2\pi} 
+ \frac{1}{\pi} t^{1/2} \left(\int_{C_0}^{t} \left| \sum_{n\leq X^2}   
\frac{\Lambda_X(n)\cos(u \log n) }{\sqrt{n} } \right|^2  \;du \right)^{1/2}
\end{align*}
By a standard mean value theorem  for Dirichlet polynomials it  is easy 
to show that if $X \ll t^{1/2}$,   the integral is $\ll \,t\, \log^2 X \,.$
Thus, writing $N_{I}(t)$ for the number of \emph{distinct} zeros that  
occur in this way, we have
$$
N_{I}(t) \leq  \frac{t}{2\pi} \log \frac{t}{2\pi} - \frac{t}{2\pi} + O(t\, \log X) \,.
$$
We will see how to take multiplicities into account below.

The second way solutions to $F_X(t)  \equiv \pi (mod\, 2\pi)$
can occur is by 
$F_X(t)$ increasing or decreasing from a  
solution $(2k+1)\pi $ and returning to this value before  
reaching the next larger or smaller odd multiple
of $\pi$. Each time this happens, there must be at least one point in between
where $F_X^{'}(t)$ vanishes. Thus, writing $N_{II}(t)$ for the number of 
\emph{distinct} zeros of $\zeta_X(s)$ arising this way, we see that 
$N_{II}(t)$  is at most big-$O$ of the number of times
$F_X^{'}(t)$ vanishes on $[C_0, t]$. 
To estimate this number we define   functions
\begin{equation*}
g_X(s) =  -\frac{\chi^{'}}{\chi}(s) - 2 \sum_{n \leq X^2} 
\frac{\Lambda_X(n)}{n^s}
\end{equation*}
and 
\begin{equation*}
G_X(s) = \frac12 \left( g_X(s) +  g_X(1-s) \right)\,.
\end{equation*}
Here we use the principal branch of logarithm on
the complex plane with the negative real axis removed.
By \eqref{arg formula 1}, $F_X^{'}(t)= G_X(\frac12+it)$, so the zeros of $F_X^{'}(t)$ on 
$[C_0, t]$ are  the zeros of $G_X(s)$ on $[\frac12+iC_0, \frac12+i t]$. 
We   bound this number by bounding
the  number of   zeros on each of the segments   $[\frac12+i   t, \frac12+ 2i t], 
[\frac12+i \frac t2, \frac12+i t], ... $, and adding. The number of zeros on any one of these
is at most the number of zeros  of $G_X(s)$ in a disk  containing the segment.  
By a familiar result  from complex analysis,   if $\D$ is a closed disk of radius $R$ centered at 
$z_0$, $f(z)$ is analytic on $\D$ with maximum modulus $M$, and  $f(z_0) \neq 0$, then there is an
absolute constant $c$ such that  $f$ has $\leq c \log(M/|f(z_0)|)$ zeros in the disc of radius
$\frac23 R$ centered at $z_0$. To apply this to the segment $[\frac12+i   t, \frac12+ 2i t]$, say,
we need a disc containing it, the maximum   of $|G_X(s)|$ on this disk, and a lower 
bound for   $|G_X|$ at the center of the disk.  We  handle the last problem first by
selecting as center a point at which we know $|G_X(s)|$ cannot be too small. The 
upper bound for $N_{II}(t)$ will follow by repeating this process for each of the segments 
and  adding the resulting estimates.

To show that one can find a satisfactory center,  fix a $\delta$ with $0< \delta <\frac12$ and set 
\begin{equation*}
\mathcal{E} (t) = \left\{ u \in[t, 2t] :  (\tfrac12 +\delta)\log\frac{t}{2\pi}\geq 
\left| f_X^{'}( u) \right|
  \geq   (\tfrac12 -\delta)\log\frac{t}{2\pi} \right\}\,.
\end{equation*}
Then  $|F_X^{'}(u)| = |G_X(\tfrac12+iu)| \leq 2 \delta \log( t/2\pi) +O(1/t)$ for all $u \in \mathcal{E} (t)$.
Now  recall  that
\begin{equation*}
 \int_{t }^{2t} \left|   f_X^{'}( u)    \right|^2  \;du
= \int_{t }^{2t} \left| \sum_{n\leq X^2}   
\frac{\Lambda_X(n)\cos(u \log n) }{\sqrt{n} } \right|^2  \;du
\ll t\log^2 X\,.
\end{equation*}
Thus, the measure of $\mathcal{E} (t)$ is  
\begin{align*}
 | \mathcal{E} (t)| = \int_{\mathcal{E} (t)} 1 \;dt \leq &
  \int_{\mathcal{E} (t)} \left( \frac{ | f_X^{'}(u) |}
  { (\tfrac12 - \delta)\log\frac{t}{2\pi} } \right)^2\;dt  \\
  \ll & \frac{t\log^2 X}{\log^{ 2} t } \,. 
 \end{align*}
It follows that there exists a constant $C_4 >0$ such that 
if $X \leq t^{C_4 }$, then $ | \mathcal{E} (t)| < \frac16 t$. 
Since the segment
$[\frac12+i (\frac32-\frac{1}{12}) t, \frac12+  i (\frac32+\frac{1}{12}) t]$
has greater length than  the set $\mathcal{E}(t)$, it  contains a point
$\frac12+it_0$ with $t_0$ not in $ \mathcal{E}(t)$, and therefore   with
$|G_X(\tfrac12+it_0)| > 2 \delta \log( t/2\pi) $.

We now let $\mathcal{D}_0(t)$ be the  closed disc of radius $t$ centered at
  $\frac12+it_0$,  and let $M$ denote the maximum of 
$|G_X(s)|$ on $\mathcal{D}_0(t)$. 
Clearly  on $\mathcal{D}_0(t)$ we have
\begin{equation*}
G_X(\s+it) \ll \log t + \sum_{n\leq X^2} 
\Lambda_X(n) n^{t-1/2} \ll X^{t+1/2}\,.
\end{equation*}
Hence, by the theorem alluded to above $G_X(s)$, has $\ll   \log(X^{t+1/2}/\delta \log t) 
\ll t\log X$   zeros inside the smaller disc $\mathcal{D}_0(\frac23 t)$ of radius
$\frac23 t$, which covers $[\frac12+i   t, \frac12+ 2i t]$. Adding estimates 
for the different intervals,  we  arrive at $\ll t\log X$ distinct zeros of $F_X^{'}(t)$.
The same bound  therefore holds for the number of distinct zeros of the second kind.

Combining the two  ways the solutions of $F_X(t)\equiv \pi \pmod{2\pi}$, or zeros of 
$\zeta_X(s)$ arise, we find that for $X \leq t^{C_4}$ there are  
\begin{equation*}
N_{I}(t) + N_{II}(t) \leq  (t/ 2\pi) \log(t/ 2\pi) - (t/ 2\pi) + O(t\, \log X)
\end{equation*}
distinct zeros.

Now, a zero $\tfrac12+i\g_X$ of $\zeta_X(\tfrac12+it)$ has multiplicity $m$ if and only if
  the first $m-1$ derivatives of  $\zeta_X(\tfrac12+it )$ with respect to $t$ vanish at $\g_X$, 
  but the $m^{th}$ does not.
It is easy to check that this is equivalent to $F_X(\g_X) \equiv \pi \pmod{2\pi},
 F_X^{'}(\g_X)= ... = F_X^{(m-1)}(\g_X) = 0$, and $F_X^{(m)}(\g_X) \neq 0$.  
Also note that our estimate for the number of zeros of the analytic function $G_X(s)$ 
counts them according to their multiplicities, and  that
$F_X^{'}(t) =  G_X(\frac12+it), F_X^{(2)}(t) =  i G_X^{(1)}(\frac12+it), 
\ldots,  F_X^{(m)}(t) =  i^{m-1}G_X^{(m-1)}(\frac12+it)$. 

Suppose then that $\tfrac12+i\g_X$ is  a zero of $\zeta_X(s)$ of the first kind and multiplicity 
$m$.  Then it is counted once in $N_{I}(t)$. Also, since the first  $m-1$
derivatives   of $F_X(t)$ vanish at $\g_X$, so does $G_X(\tfrac12+it)$ and its first $m-2$
derivatives. Thus, $\tfrac12+i\g_X$ is counted another $m-1$ times in $N_{II}(t)$, and 
therefore with the correct multiplicity in $N_{I}(t) + N_{II}(t)$.

Next suppose that  $\tfrac12+i\g_X$ is a zero of multiplicity $m$ of the second kind. 
Then it is counted  at least once in $N_{II}(t)$ because  $F_X^{'}(t) = G_X(\tfrac12+it)$ 
vanishes at a nearby point. Also, at $\g_X$ itself we  have $F_X^{'}(\g_X)= \cdots = 
F_X^{m-1}(\g_X)=0$, and $F_X^{m}(\g_X) \neq 0$. This means that $G_X(s)$ and its
first $m-2$ derivatives are zero at $\tfrac12+i\g_X$, so this point is counted $m-1$
times by $N_{II}(t)$. Thus, zeros of the second kind with multiplicity $m$ are counted
with weight at least $m$ in $N_{II}(t)$.   
 
We now see that 
$$
N_{X}(t) \leq  \frac{t}{2\pi} \log \frac{t}{2\pi} - \frac{t}{2\pi} + O(t\, \log X) \,.
$$

Both assertions of the theorem now follow from this and the lower bound in \eqref{lower bd}. 
\end{proof}

\section{The number of simple zeros of $\zeta_X(s)$}  

We saw in the last section  that a zero $\tfrac12+i\g_X$ of $\zeta_X(s)$ is 
simple if and only if  $F_X(\g_X) \equiv \pi \pmod{2\pi}$ and $F_X^{'}(\g_X) \neq 0$. 
Let $ N_X^{(1)}(t)$ denote the number of such zeros up to height $t$.
From  \eqref{F'} and \eqref{pos deriv}
we see that  $F_X^{'}(\g_X) > 0$ if 
 $X$ is not too large, and therefore that $\tfrac12+i\g_X$ is a simple
 zero of $\zeta_X(s)$. Combining this with Theorem~\ref{exact zero count}, we obtain
 
  \begin{thm}\label{simple zeros} Assume the Riemann Hypothesis.
 There exists a constant $C_3>0$ such that  if $X <  \exp\big(C_3 \log t /  \Phi(t) \big)$,
 then all the zeros  of $\zeta_X(\tfrac12+it)$ with $t \geq C_0$
 are simple and
$$
N_X^{(1)}(t)= N_X(t) = 
\frac{t}{2\pi} \log \frac{t}{2\pi} -\frac{t}{2\pi}  
 - \frac{1}{\pi}\sum_{n\leq X^2}\frac{\Lambda_X(n) \sin(t\log n)}
 {n^{\frac12}\log n} +O_X(1)\,. 
 $$
\end{thm}
 
As in our results for  $N_X(t)$,  the condition on $X$ is  almost certainly too restrictive. 
The following unconditional but less precise result  is valid for larger $X$.

\begin{thm}\label{simple zeros 2} Let
$\e >0$ and $X \leq \exp(o(\log^{1-\e} t))$. 
Then as $t \to \infty $, the number of simple zeros up to height 
$t$ is 
$$
 N_X^{(1)}(t) = \left(1 + o(1)\right) \;\frac {t}{2\pi} \log \frac{t}{2\pi}  \,.
 $$
\end{thm}
  
\begin{proof} Let $N$ be the number of zeros of $\zeta_X(\frac12+iu)$
in $[t, 2t]$ and $N^{*}$ the number of these that are multiple.
By Theorem~\ref{no. of zeros}, there is a constant $C_4$ such that if 
$X \leq t^{C_4}$, then $N \ll t \log t$. We may therefore split the  
$N^{*}$ multiple zeros into  $K \ll \log t$ sets $\mathcal{S}_1, \mathcal{S}_2,
 ..., \mathcal{S}_K $, in each of which the points are at least $1$ apart. 
Let $\mathcal{S}$ be one of these sets and let $\g_1, \g_2, ... \g_R$ be its points.  
Then these must all satisfy
\begin{equation*}\label{F' again}
0= F_{X}^{'}(\g_r) = \log \frac{\g_r}{2\pi} - 2\sum_{n\leq X^2} 
  \frac{\Lambda_X(n) \cos(\g_r \log n)}{n^{\frac12}  } 
  + O\left(\frac{1}{\g_r}\right)\,.
\end{equation*}
Writing
\begin{equation*}
 \sum_{n\leq X^{2k}}  \frac{A_X(n)}{n^{\frac12+iu}  } =
\left(  \sum_{n\leq X^2} \frac{\Lambda_X(n)}{n^{\frac12+iu}  } \right)^k  \,,
\end{equation*}
we have  by a mean value theorem of Davenport (Montgomery~\cite{M1}) 
\begin{equation*} 
\sum_{r=1}^{R} \left|  \sum_{n\leq X^2} 
  \frac{\Lambda_X(n) }{n^{\frac12+i \g_r}  }  \right|^{2k}
  \ll \left( t + X^{2k}  \log (X^{2k})\right) \log (X^{2k}) \sum_{n\leq X^{2k}} 
   \frac{|A_X(n)|^2}{n}\,.
\end{equation*}
It is not difficult to show that the sum on the right is
$\ll_{2k}  \log^{2k} X$ so,  if $X \leq t^{1/2k}$, 
the   right-hand side is $\ll_k  t \log^{2k+2} X$. 
(We also require $X \leq t^{C_4}$,  so we assume that
$k \geq C_{4}/2$.)  
On the other hand, by \eqref{F' again} the left-hand side must be $ \gg R \log^{2k} t$. 
Therefore $|\mathcal{S}| = R \ll_{k} t (\log^{2k+2} X /  \log^{2k} t)$. There are  
$K \ll \log t$\;  sets $\mathcal{S}_k$, so the total possible number 
of multiple zeros is  $N^{*} \ll _{k} t (\log^{2k+2} X /  \log^{2k-1}t)$.  
This is $o_k (t \log t)$ if $X \leq \exp\left(o\left(\left(\log t\right)^{1-1/(k+1)}\right)\right)$. Taking
$k$ large enough so that $1/(k+1) < \e$, we obtain the result.
\end{proof}


\section{The relative sizes of $\zeta_X(s)$ and $\zeta(s)$ and the
relation between their zeros} 

Although we have not proved that $\zeta_X(s)$ 
approximates  $\zeta(s)$  pointwise when $\s$ is very close   
to $\frac12$,  the similarity between
the formulae for $N_X(t)$ and $N(t)$ suggests there might 
be a close relationship between the two 
functions even on the critical line. Indeed, comparing the  
graphs  of   $|\zeta_X(\frac12+it)|$   and $|\zeta(\frac12+it)|$  
for a wide  range  of $X$ and $t$  (see  Figures~\ref{function} and \ref{fnc}), 
one is struck by two things:   
\newpage
\begin{enumerate}
\item  the   zeros of $\zeta_X(\frac12+it)$ are quite close 
to those of $\zeta(\frac12+it)$, even for relatively small $X$, and  \\
 \item   \noindent as $X$   increases,   $|\zeta_X(\frac12+it)|$ seems  
to approach  $2 |\zeta(\frac12+it)|$.
 \end{enumerate}

\begin{figure}\label{function}
\centering 
\includegraphics[width=6.5in]{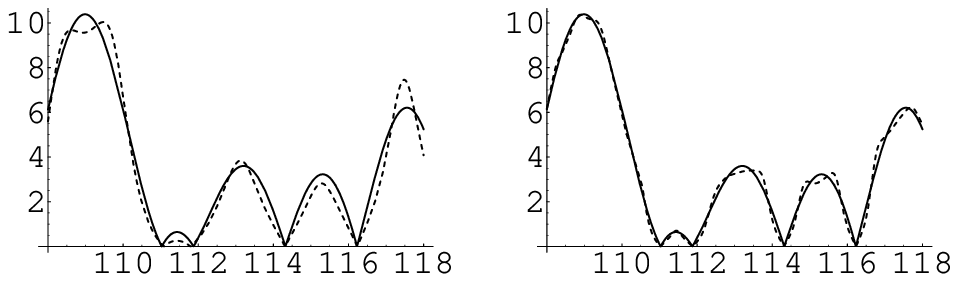}
\caption{Graphs of $2|\zeta(\frac12+it)|$ (solid) and 
$  |\zeta_X(\frac12+it)|$ (dotted) near $t=114$ for $X=10$ 
and $X=300$, respectively.}\label{function}
\centering 
\includegraphics[width=6.5in]{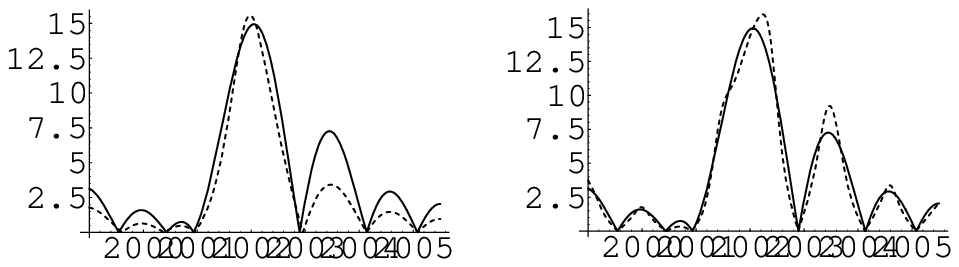}
\caption{Graphs of $2 |\zeta(\frac12+it)|$ (solid) and 
$  |\zeta_X(\frac12+it)|$ (dotted) near $t=2000$ for $X=10$ 
and $X=300$, respectively.}\label{fnc}
\end{figure}

An explanation for the second observation  is that  
although $\zeta_X(s)$ approximates $\zeta(s)$ to the right of the
critical line, so does $P_X(s)$. Therefore 
$\zeta_X(s)= P_X(s)+\chi(s) P_X(\overline{s})$ might be a closer 
approximation to the function   
$
\F(s) =  \zeta(s) +\chi(s) \zeta(\overline{s})  
$
than  to $\zeta(s)$.  If  this is the case, then
on the critical line we have by the functional equation   that
\begin{align*}
\zeta_X(\tfrac12+it) \approx\F(\tfrac12+it) = &\zeta(\tfrac12+it) +\chi(\tfrac12+it) \zeta(\tfrac12-it) \\
 = &\zeta(\tfrac12+it) +|\chi(\tfrac12+it)|^2 \zeta(\tfrac12+it) \\
 = & 2 \,\zeta(\tfrac12+it) \,.
\end{align*}

To establish both observations rigorously we need to introduce
a slightly  modified version of $\zeta_X(s)$.
Let
\begin{equation*}
P_X^{*}(s) = P_X(s) \exp\left(-F_2\left((s-1)\log X \right)\right)
\end{equation*}
and define
\begin{equation*}\label{P*}
\zeta_X^{*}(s) = P_X^{*}(s) + \chi(s) P_X^{*}(\overline{s}) \,.
\end{equation*}
Note that by  \eqref{F_2 of pole}, 
\begin{equation*}
P_X^{*}(s) = P_X(s)  \exp \left( O\left(  \frac{X^{2-2\s}}{\t^2  \log^2 X}  \right)  \right)
\end{equation*}
when $\s \geq  \frac12$ and $|s-1| \geq \frac1{10}$. The difference between 
$P_X(s)$ and $P_X^*(s)$, and so also $\zeta_X(s)$ and $\zeta_X^*(s)$, is 
 small  when $2 \leq X \leq t^2$, as we have been assuming till now.  
We  need to take $X$ much larger, though,    in what follows. 
Similarly, we replace  $F_X(t)$ by  
\begin{align}\label{F_X^*}
F_X^{*}(t) = & -\arg\chi(\tfrac12+ i t) + 2 \arg P_X^{*}( \tfrac12+ i t)    \notag \\ 
=& -  \arg\chi(\tfrac12+ i t) + 
2\bigg( \arg P_X(\tfrac12+it)  -  \Im \,F_2\left( (-\tfrac12 + i t ) \log X \right)   \bigg) \,.
\end{align}
By  \eqref{arg formula 1}    and  \eqref{F_2 of pole}   
$$
F_X^{*}(t)  = F_X(t) + O\left(  \frac{X}{\t^2 \log X}  \right)\,,
$$ 
so  these  two functions are also close when  $X \leq t^2$.
The zeros of $\zeta_X^*(\frac12+it)$ are the solutions of $F_X^{*}(t)
\equiv \pi \pmod{2\pi }$ and we will show that (1) and (2) above hold 
provided we use $\zeta_X^*(\frac12+it)$ in place
of $\zeta_X(\frac12+it)$.

Assume  the Riemann Hypothesis is true.  Taking the  argument  of both sides of  
\eqref{zeta form 1} and recalling  that $S(t) =(1/\pi) \arg \zeta(\frac12+it)$, we see that
\begin{equation}
\begin{aligned}\label{S(t) 2}    
\pi S(t) = \arg P_X(\tfrac12+it) 
 - & \Im \,F_2\left((-\tfrac12 + i t ) \log X \right)      
 +  \sum_{\g}\Im \;F_2 \left(i(t-\g)\log X \right)    \\     
    & + O\left(  \frac{X^{-\frac32}}{\t^2 \log X}  \right)\,, 
\end{aligned}
\end{equation}
where $\g$ runs through the ordinates of the zeros of $\zeta(s)$.
We use this to replace the quantity in parentheses in \eqref{F_X^*} and obtain
\begin{align*}\label{F* 10} 
F_X^{*}(t)  = -  \arg\chi(\tfrac12+ i t) + 
2  \pi S(t) - 2 \,\Im \sum_{\g} F_2 \left(i(t-\g)\log X \right)         
    + O\left(  \frac{X^{-\frac32}}{\t^2 \log X}  \right)   
\,.
\end{align*}
Now, by \eqref{N(t) 2} 
$$
-  \arg\chi(\tfrac12+ i t) + 2  \pi S(t) =2\pi N(t) -2\pi,
$$
thus
\begin{equation}\label{congruence 2} 
\frac{1}{2\pi}F_X^{*}(t)  =    N(t) -1- \frac{1}{\pi} \;\Im \sum_{\g} F_2 \left(i(t-\g)\log X \right)   
 + O\left(  \frac{X^{-\frac32}}{\t^2 \log X}  \right)   \,.
\end{equation}
 
We  use this first to show  that the zeros of $\zeta_X^*(s)$ cluster 
around the zeros of $\zeta(s)$
as $X \to \infty$. 
Let $  \g$ and $ \g^{'}$ denote   ordinates of distinct consecutive zeros 
of $\zeta(s)$, and set $\Delta = | \g  - \g^{'} |$. 
Also, fix an $\e$ with $0<\e <1/4$ and  let $\mathcal{I}=
[\g  +\e \Delta,\,\g^{'}  - \e \Delta]$.
Then by \eqref{F_2 formula 2} and Lemma~\ref{Zero Sum},  
 if $X \geq \exp(1/\e \Delta)$    we have
\begin{align}\label{F_2 is small}
\sum_{\g}F_2 \left(i(t-\g)\log X \right)   
\ll & \frac{1}{\log^2X}\sum_{|\g-t| > \e \Delta} \frac{1}{ (t-\g)^2} \\ 
\ll & \frac{1}{\e \Delta \log^2X} \left(\log \t  + \frac{\Phi(\t)}{\e\Delta} \right)  \notag
\end{align}
 uniformly for $t \in \mathcal{I}$.
It now follows from \eqref{congruence 2}   that given any $\delta>0$, there exists an 
$X_0 = X_0(\g,  \e, \delta)$ such that if 
  $X \geq X_0$, then   
\begin{equation}\label{F* approx 10}              
\left|\left| \frac{F_X^{*}(t)}{2\pi}  - N(t) \right|\right|  < \delta    \,,
\end{equation}
uniformly for $t \in \mathcal{I} $. Here $|| x ||$ denotes distance to the nearest integer.
Since  $N(t)$ is an integer when $t \in \mathcal{I}$, this means that if $\delta <\frac12$,
then $\frac12+it$ is \emph{not} a zero of $\zeta_X^{*}(s)$. Thus $\mathcal{I}$ is free from
zeros of $\zeta_X^{*}(s)$ when $X$ is sufficiently large. 

Now we show that $|\zeta_X^{*}(\frac12+it)|$ tends to 
$2|\zeta(\frac12+it)|$ on $\mathcal{I}$. By   \eqref{P*}
and \eqref{F_X^*} we may write 
\begin{equation*}\label{zeta* form 10}
\zeta_X^{*}(\tfrac12+it) 
= P_X^{*}(\tfrac12+it) \left( 1+ e^{-i F_X^{*}(t)}  \right) .
\end{equation*}
Also, by Theorem~\ref{thm on zeta} and the definitions of $P^{*}_X$ and
$Z_X$, we have
\begin{equation*}\label{zeta form 10}
\zeta(s) = P_{X}^{*}(s) \exp\left(\sum_{\g}F_2 \left(i(t-\g)\log X \right)\right)
 \left(1+  O\left(  \frac{X^{-\frac52  }  }{\t^2 \log^2  X}  \right) \right)\,.
 \end{equation*}
 From the first of these we see that if \eqref{F* approx 10} 
 holds with $\delta$ sufficiently small, then   
 $$
|\zeta_X^{*}(\tfrac12+it) |
= |P_X^{*}(\tfrac12+it)| \big( 2+ O(\delta) \big)
$$
uniformly for $t \in \mathcal{I}$. From the second and \eqref{F_2 is small} we see that
if $X$ is large enough, then
$$
|\zeta(\tfrac12+it) |
= |P_X^{*}(\tfrac12+it)| \big( 1+ O(\delta) \big)
$$
on $\mathcal{I}$. Thus, $|\zeta_X^{*}(\tfrac12+it)| \to 2 |\zeta(\tfrac12+it)|$ 
as $X \to \infty$ uniformly for $t \in \mathcal{I}$.

Combining our results we now have 
\begin{thm} Assume the Riemann Hypothesis. Let $\g$ and $\g^{'}$ denote ordinates of distinct consecutive zeros of the 
Riemann zeta-function, and let $\mathcal{I}$ denote a closed subinterval of
 $(\g, \,\g^{'})$. Then for all $X$ sufficiently large  $\zeta_X^{*}(\frac12+it)$ has no
 zeros in $\mathcal{I}$. Moreover, $|\zeta_X^{*}(\tfrac12+it)| \to 2 |\zeta(\tfrac12+it)|$ 
as $X \to \infty$ uniformly for $t \in \mathcal{I}$.
 \end{thm}
  
I hope to give a more complete analysis of the approximations above in a subsequent
article. 
 
E. Bombieri has pointed out to me 
that  \eqref{S(t) 2}  is closely related to an explicit 
formula  of  Guinand (\cite{Gu 1}, \cite{Gu 2}), namely,
\begin{align*} 
\pi  S(t) &=\\
 &- \lim_{X \to \infty}  \left( \sum_{n \leq X }\frac{\Lambda (n)\sin(t \log n)}
 {n^{\frac12}\log n}  - \int_{1}^{X}   \frac{ \sin(t\log u)}{u^{\frac12} \log u} du  
  - \frac{\sin(t\log X)}{\log X} \left( \sum_{n\leq X } \frac{\Lambda(n) }{n^{\frac12} } 
  -2 X^{\frac12}  \right)\right)  \\
 & -\frac78 \pi  + \frac{1}{2} \left(\arg\Gamma(\tfrac12 +it) -t\log t + t  \right)
 + \arctan 2t -\tfrac14 \arctan(\sinh \pi t) \,.
\end{align*}  
There is no sum over zeros here because Guinand is taking a limit. Also,
the $\Lambda(n)$ are unweighted. However, this is only a minor difference.
 
It is remarkable that  the zeros of  $\zeta_X^*(s)$ 
and $\zeta_X(s)$  are  close to those of  $\zeta(s)$
(Figures~\ref{function}  and \ref{fnc}) even when $X$ is small. 
Formula \eqref{congruence 2} offers a possible explanation for this. 
Suppose that $t=\g$, the  ordinate of a zero of $\zeta(s)$
with multiplicity $m$.  
Then $N(\g)$ is an  integer, so by \eqref{congruence 2} 
\begin{equation*}
\frac{1}{2\pi}F_X^{*}(\g)  \equiv   - \frac{1}{\pi} 
\;\Im \sum_{  \g^{'} } F_2 \left(i(\g-\g^{'})\log X \right)   
+  O\left(\g^{-2} \right)   \pmod{1}  \,.
\end{equation*}
Now, a  more  precise version of \eqref{F_2 formula 1} 
is  that if $y$  is real,
$$
\Im \,F_2(iy) = \arg iy + \sum_{k=0}^{\infty} a_k\, y^{2k+1},
$$
where the $a_k$ are real and the argument is $\pi/2$ when  $y= 0$
(the limit as $y \to 0^{+})$.
This is an odd function (for $y\neq 0$). Furthermore, for larger $y$
we have $ \Im\, F_2(iy) = \frac{\sin y}{y^2}(1+O(1/|y|))$ by \eqref{E_2 formula 2}.
Thus, the $m$ terms in the sum with $\g^{'}=\g$ contribute  $m\,\pi/2$, and the terms with  
$|\g-\g^{'}| \log X$ large are decreasing and oscillating. It might also be the case
that small and  intermediate range terms  cancel  out to a large degree 
because  $\Im\, F_2(iy)$ is odd and we expect the $\g^{'}$s to be somewhat random. 
If this is so, then $(1/2\pi) F_X^{*}(\g) $ will be close to $  m/2 \pmod{1}$. Thus, if $m$
is odd (it is believed that $m$ always equals $1$) it would not be surprising to 
find a zero of $\zeta_X^{*}(\frac12+it)$ nearby.



While writing this paper, I learned from a lecture by J. P. Keating
that he and E. B. Bogomolny  had  worked with a function similar to  $\zeta_ {t/2\pi}$   
restricted to the critical line as a heuristic tool to   calculate the  
  pair correlation function of the zeros of  $\zeta(s)$ (see,  for example, 
Bogomolny and Keating~\cite{BK} and Bogomolny~\cite{B}). In fact
Professor Keating~\cite{K}   had first considered such a function in the early 90s and 
observed that its zeros are quite close to those  of the zeta-function. He and his 
graduate student, Steve Banham, also heuristically investigated how close the  
zeros of $\zeta_ {X}(\frac12+it)$ and $\zeta(s)$ are as a function  of $X$.   


\section{Why  are the zeros of $\zeta_X(s)$ simple  and why do they repel?}

The construction,  properties, and graphs of the functions $\zeta_X(s)$
 suggest that they model the behavior of  
$\zeta(s)$, particularly with regard to the position of zeros.  
Therefore, explanations of why  the zeros of $\zeta_X(s)$
are  simple and  repel each other could shed light on why 
the zeros of $\zeta(s)$ have these same properties.

 Theorem~\ref{simple zeros} shows that if the Riemann Hypothesis holds, then 
the zeros of $\zeta_X(\tfrac12+it)$ with $t\geq C_0$ are simple provided that 
$X \leq \exp(C_3 \log t/\Phi(t))$ for some constant $C_3>0$. Futhermore 
 Theorem~\ref{simple zeros 2} shows  unconditionally that even for $X$ as large 
 as  $\exp(o(\log^{1-\e} t))$, $100\%$ of the zeros are simple.  
The structure of $\zeta_X(\tfrac12+it)$ suggests why.

The zeros of  $\zeta_X(\tfrac12+it)$ 
for $t \geq C_0$ are  the solutions of the congruence 
$F_X(t) \equiv \pi \pmod{2\pi}$. In other words, they are the 
$t$-coordinates of the points where the curve $y =  F_X(t)$ crosses the equally 
spaced horizontal lines $y=(2k+1) \pi$. If such a $t$ is to be the ordinate of  
a multiple zero of $\zeta_X(\tfrac12+it)$, it   also
has to be a solution of the equation $F_X^{'}(t) =0$. We saw that this cannot happen
for $X \leq \exp(C_3 \log t/\Phi(t))$ and that it cannot happen often if $\log X =o(\log \t)$.
But clearly, even for  $X$ a power of $t$ this should   happen rarely, if ever.

What about repulsion?
By  \eqref{F'} and   Theorem~\ref{Bound on real part sum},  
$F_{X}^{'}(t) \ll \Phi(t)\log X+ \log t $ when $2 \leq X \leq t^{2}$.
As in Section 7,   we divide the zeros into two kinds. The first  kind come  about by $y=F_{X}(t)$  
increasing or decreasing from $y= (2k+1)\pi$ to the next 
larger or smaller odd multiple of $\pi$  without first  re-crossing $y=(2k+1)\pi$. 
All other zeros   are  zeros of the second kind. 
Suppose that  $\g_X$ and$\g'_X$ are ordinates of  consecutive zeros of $\zeta_X(\tfrac12+it)$,
and  $\frac12+i\g_X$ is a zero of the first kind. Then
$F_{X}(\g^{'}_X) - F_{X}(\g_X)=\pm2\pi$ and we have
\begin{align*}
2\pi =& |F_{X}(\g^{'}_X) - F_{X}(\g_X)| = \left|\int_{\g_X}^{\g'_X} F_{X}^{'}(u)\,du \right| \\
\ll & (\g_X^{'} - \g_X )\left(  \log \g_X  +  \Phi(\g_X)\log X \right) \,.
\end{align*}
Thus,
\begin{equation*}
 \g_X^{'} - \g_X  \gg \frac{1}{ \log \g_X  +  \Phi(\g_X)\log X  }      \,. 
\end{equation*}
Recall that  
$(\log \g_X /\log\log \g_X )^{1/2} \ll \Phi(\g_X) \ll \log \g_X$.
Thus, if $X \leq \g_X^2$, then
\begin{equation}\label{zeros gap 1}
 \g_X^{'} - \g_X   \gg 1/\log^{a} \g_X
\end{equation}

for some $a \in [ 1,\, \frac32]$.

Note that if $X \leq \exp(C_3\log t/\Phi(t))$ with $C_3$ as in 
Theorem~\ref{simple zeros}, then $F_{X}^{'}(t)>0$ and all zeros are of the first
kind. Furthermore, by the proof of Theorem ~\ref{no. of zeros}, $\sim (t/2\pi) \log(t/2\pi)$ 
of the zeros are of the first kind when $\log X=o(\log \t)$.

If $\frac12+i\g_X$ is a zero of the  second kind, then  $F_{X}(\g^{'}_X) - F_{X}(\g_X)=0$  
and the  argument above does not work. It may be, however,  
that this does not happen often, that is, that most zeros are of the first kind. 

\begin{figure}\label{sawtooth}
\centering 
\includegraphics[width=6.5in]{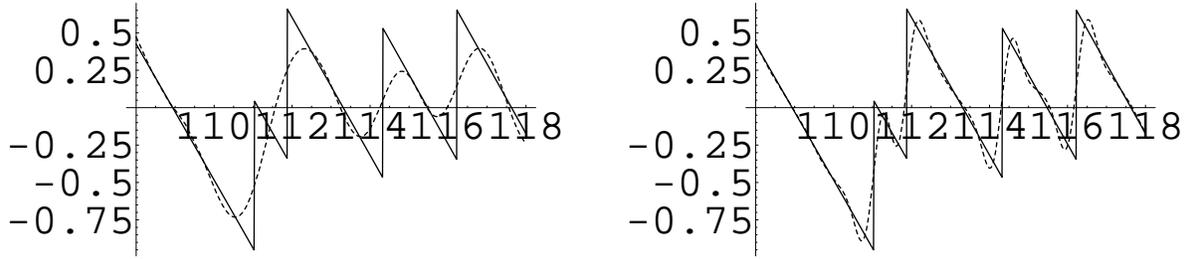}
\caption{Graphs of $S(t)$ (solid) and $(-1/\pi)f_X(t)$ (dotted) near $t=114$ for $X=10$ 
and $X=300$, respectively.}\label{sawtooth}
\end{figure}
 
To see why first observe that $S(t)$ is a saw-tooth function because $N(t)$ is a  step-function
consisting of  the increasing function  $(t/2\pi)\log (t/2\pi) -(t/2\pi) +7/8 +O(1/\t)$ 
plus  $S(t)$. Now between  consecutive ordinates $\g, \g^{'}$
of zeros of $\zeta(\frac12+it)$, $S(t)$ decreases essentially  linearly
with slope $-(1/2\pi)\log (\g/2\pi)$; it then  jumps  at $\g^{'}$  by 
an amount equal to the multiplicity of the zero $\frac12+i\g^{'}$.
The heuristic argument at the end of the last section suggesting 
that $\sum_{\g}F_2(i(t-\g)\log X)$ is usually small away from ordinates of 
zeta zeros,  when applied to  \eqref{S(t) 2} with $X \leq t^2$, implies that between ordinates
$$
 S(t)  \approx  -\frac{1}{\pi} f_X(t) = 
 -\frac{1}{\pi} \left( \sum_{n\leq X^2} 
\frac{\Lambda_X(n) \sin(t\log n)}{\sqrt{n} \log n} \right).
$$ 
Of course, it is not clear how large $X$ should be relative to $t$.
However, graphs of  $f_X(t)$ indicate that they are close to the graph of 
$S(t)$ when $X$ is moderately large,  there are small oscillations 
along the downward
slopes of $S(t)$, and then a  flatter, not necessarily vertical, rise
near the jumps of $S(t)$ (Figure~\ref{sawtooth}). 
For  
$$
\frac{F_{X}(t)}{2\pi} \approx 
\frac{t}{2\pi}\log\frac{t}{2\pi} -\frac{t}{2\pi} -\frac18  
-\frac{1}{\pi} f_X(t)  \,,
$$
which approximates $N(t)-1$, this means that
the oscillations tend to be along the flat part 
of the ``steps'' and not at the rise (Figure~\ref{stairs}). 
\begin{figure} 
\centering 
\includegraphics[width=6.5in]{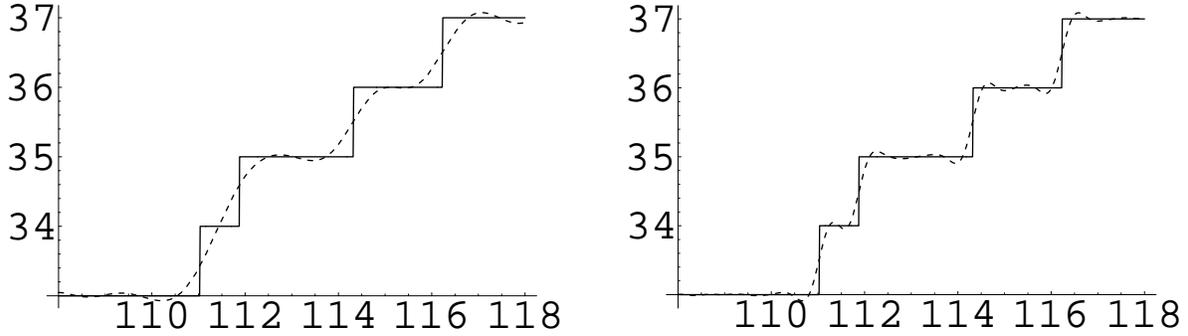}
\caption{Graphs of $N(t)$ and $(1/2\pi)F_X(t)+1$ near $t=114$ for $X=10$ 
and $X=300$}\label{stairs}
\end{figure}
However, 
zeros of   $\zeta_X(\frac12+it)$ correspond to solutions of  
$ F_{X}(t)/2\pi \equiv \frac12 \pmod{1}$, and these will be 
abscissae of points that are about half-way up the rise of $ F_{X}(t)$.
This would suggest that  zeros of the second kind are unlikely.

Our arguments have assumed that $X \leq t^2$, but we do not know
whether this is  appropriate for imitating the zeta-function in this context. 
If not, we could repeat the arguments with  $ F_{X}^{*}(t)$. This would introduce the term 
$- \Im \,F_2\left((-\tfrac12 + i t ) \log X \right)$, which can be as large as
$X/t^2 \log^2X$. Applied to the argument for gaps between zeros 
of the first kind with ordinates  around $t$, and assuming $X$ 
is   a power of $t$   greater than $2$, this   leads to  
\begin{equation*}\label{zeros gap 2}
 \g_X^{'} - \g_X   \gg 1/ \g_X ^{b}
\end{equation*}
for some positive $b$ in place of  \eqref{zeros gap 1}. The repulsion between 
zeros of the zeta-function obtained by   extrapolating from Montgomery's
pair correlation conjecture predicts that 
\begin{equation*}
 \g_X^{'} - \g_X   \gg 1/ \g_X ^{\frac13+\e}.
\end{equation*}
 

\section{Other L-functions and  sums of L-functions}

The ideas above have obvious extensions to the more general setting of the 
Selberg Class  of L-functions~\cite{S} and similar classes of functions with 
Euler products and functional equations, such as that defined by Iwaniec and 
Kowalski~\cite{IK}.  Here we briefly indicate how this 
looks for  Dirichlet L-functions.
We then  consider the analogue of the problem of the distribution of zeros of 
linear combinations of L-functions. This approach provides a new heuristic 
explanation for why such combinations should have almost all their zeros on the 
critical line.

Let $L(s, \chi)$ denote the Dirichlet L-function 
with character $\chi$ modulo $q$ and  
functional equation
\begin{equation}\label{L fnc equ}
 L(s, \chi) =  
\frac{\t(\chi)}{i^{\mathfrak{a}} \sqrt{q}  } \Psi(s) L(1-s, \overline{\chi}) \,,
\end{equation}
where
\begin{equation}\label{Psi}
\Psi(s) =  
 \left( \frac{\pi}{q}\right)^{s- \frac12}
 \frac{\Gamma\left(\frac{1+\mathfrak{a}-s}{2}\right)}{\Gamma\left(\frac{ \mathfrak{a}+s}{2}\right)}\,,
\end{equation}
 $\t(\chi)$ is Gauss' sum,
and  $\mathfrak{a} =0$ or $1$ according to whether $\chi(-1) =1$ or $-1$.
When $q=1$, this is the functional equation for $\zeta(s)$ because 
$\t(\chi)=1$ and $\mathfrak{a}=0$. 
Since the factor preceding $\Psi(s)$ has modulus $1$, 
we may  rewrite \eqref{L fnc equ} as 
\begin{equation*}
e^{i \a}  L(s, \chi) =  
  \Psi(s) e^{-i \a} L(1-s, \overline{\chi}) \,,
\end{equation*}
where $  \a =\a(\chi) \in \mathbb{R}.$

We now define the functions
$$
L_X(s, \chi) = P_X(s, \chi) + e^{- 2 i \a} \Psi(s) P_X(\overline{s}, \overline{\chi} ) \,,
$$
where 
$$
P_X(s, \chi) = \exp \left( \sum_{n \leq X^2}
\frac{\Lambda_X(n) \, \chi(n)}{n^s \log n}  \right)  \,.
$$
Observe that $P_X(\overline{s}, \overline{\chi} )=  \overline{P_X(s,  \chi)}$.
Clearly   theorems corresponding to those we have proved for $\zeta_X(s)$  hold  for $L_X(s, \chi)$. In particular, one can show that all zeros in 
$\frac12 <\s \leq 1$ have
 imaginary part $\leq C_0$. Further, if the Riemann Hypothesis holds for 
 $L(s, \chi)$, then  $N_X(t, \chi)$, the number of zeros
 of  $L_X(\frac12+iu, \chi)$ with $0\leq u \leq t$, satisfies
\begin{equation*}
N_X(t, \chi) \geq \frac{t}{2\pi} \log  \frac{qt}{2\pi} - \frac{t}{2\pi} +O(\log  \t)
\end{equation*}
when $2\leq X \leq t^2$, and equality holds if $X$ is much smaller. (We assume $q$
is fixed.)  Also,  unconditionally we have 
 $N_X(t, \chi) =(1+o(1))  (t/2\pi) \log (qt/2\pi) $ provided 
 $\log X =o(\log \t).$

A number of authors (\cite{BH 1}, \cite{BH 2}, \cite{He}, \cite{S}) have studied  the  location of  
zeros of  linear combinations of the type
$$
\mathcal{L}(s) =  \sum_{j=1}^{J} b_j e^{i \a_j} L(s, \chi_{j}),  
$$
for Dirichlet and other  L-functions with the ``same'' functional 
equation, that is, having the same factor $\Psi(s)$.  Here the $b_j$'s
are real and non zero and the inclusion of the factors $e^{i \a_j}$
ensures that  $\mathcal{L}(s)$ satisfies
$$
\mathcal{L}(s) =\Psi(s)  \overline{\mathcal{L}(1-\overline{s}) }\,.
$$
Typically, $\mathcal{L}(s)$
has infinitely many zeros off the critical line  
but no  Euler product. 
 Bombieri and Hejhal (\cite {BH 1},\cite {BH 2}) have shown, however,  that 
if the Riemann Hypothesis holds for each of the L-functions and 
their zeros satisfy a plausible spacing hypothesis, then  
$100\%$ of their zeros (in the sense of density) are on the line.  
In the case of Dirichlet and certain other L-functions, Selberg (unpublished) has shown unconditionally 
that such combinations have a  positive proportion of their zeros on the critical line.

The idea leading to these results  was first suggested by 
H. Montgomery and is  roughly  as follows. 
Consider the case of two distinct Dirichlet  
L-functions $L(s, \chi_{1})$ and 
$L(s, \chi_{2})$ to the same modulus $q$ and having the same functional equation. 
One can show  that $f_1(t)= \log |L(s, \chi_{1})|/\sqrt{\pi \log\log t}$ and 
$f_2(t)=\log |L(s, \chi_{2})|/\sqrt{\pi \log\log t}$ behave like independent normally distributed random variables with mean $0$ and standard deviation $1$. Thus,  asymptotically half the time on $[T, 2T]$  we should expect  the first function to be much larger than the second, and the other half of  the time much smaller. It can also be shown that  in any  interval $\mathcal{I}$ of length $\exp(\sqrt{\log\log T})/\log T$, one function dominates the other except possibly on a subset of measure $o(\mathcal{I})$.  
Suppose then that  $f_1(t)$ dominates in  $\mathcal{I}$.
Then $|b_1 e^{i \a_1} L(s, \chi_{1}) + b_2 e^{i \a_2} L(s, \chi_{2})|$ is essentially the size of $|b_1 L(s, \chi_{1})|$ and, if the zeros of $L(s, \chi_{1})$ are well-spaced, the zeros of $b_1  e^{i \a_1} L(s, \chi_{1}) + 
b_2  e^{i \a_2} L(s, \chi_{2})$ will be perturbations of the  zeros of  
$L(s, \chi_{1})$. Thus, if all or almost all of the zeros of each L-function  is on the  critical line, almost all the zeros of the sum should be also.
 

We now ask what happens if we replace  each $L$-function  in the linear combination  by the corresponding function $L_X(s, \chi_{j})$.
Set
\begin{align*}
 \mathcal{L}_X(s) =  \sum_{j=1}^{J} b_j e^{i \a_j} L_X(s, \chi_{j})\,,  
\end{align*}
where the $b_j$ are in $\mathbb{R}-\{0\}$,
and let $\mathcal{N}_X(t)$ denote the number of zeros of
$ \mathcal{L}_X(s)$ on $\s=\frac12$ up to height $t$. Using the 
definition of $\mathcal{L}_X(s)$
we  write this as 
\begin{align*}
 \mathcal{L}_X(s) = & \sum_{j=1}^{J} b_j e^{i \a_j} P_X(s, \chi_{j})
  +  \Psi(s)    \overline{ \sum_{j=1}^{J} b_j e^{i \a_j} P_X(s, \chi_{j}) }
  \\
  = & \mathcal{P}_X(s)  +  \Psi(s) \overline{\mathcal{P}_X(s)}  \,.
\end{align*}
Clearly
 $\mathcal{L}_X(s)$ has zeros on $\s=\frac12$ if  either  
\begin{enumerate}
\item 
 $
\quad \mathcal{P}_X(\tfrac12 +it) =0 \,,
 $
\\
 or
 \item 
 $
 \quad \mathcal{F}_X(t) =\arg \Psi(\tfrac12+it) -2 \;\arg \mathcal{P}_X(\tfrac12+it)  
  \equiv \;\pi \pmod{2\pi}\;.
$ 
\end{enumerate}
For the moment let us pass over the first case and 
count the number of points at which the second case happens  
but the first does not.
By
\eqref{Psi}    
\begin{equation*}
\arg\Psi(s) = -t\log \frac{tq}{2\pi}+t  - c_0  +O(\frac{1}{\t}) \,,
\end{equation*}
with $c_0$ a real number.
Thus,  $(2)$ happens at least  
\begin{equation*}\label{zero lower bd 2}
  \frac{t}{2\pi}\log \frac{tq}{2\pi} - \frac{t}{2\pi} 
  -2 \;\arg \mathcal{P}_X(\tfrac12+it)   +O(1) 
\end{equation*}
times on $[0, t]$. Here we define $\arg \mathcal{P}_X(\tfrac12+it)$  
by continuous variation from some point $\s_0>1$ on the real axis
up to $\s_0+it$ and then over to $\frac12+it$, with our usual convention
if $\mathcal{P}_X$ vanishes at $\frac12+it$. To bound $\arg \mathcal{P}_X(\tfrac12+it)$,
the point $\s_0$ requires some consideration. For each $j$ write
$$
P_X(s, \chi_{j}) =
\exp \left( \sum_{n \leq X^2}
\frac{\Lambda_X(n) \, \chi_{j}(n)}{n^s \log n}  \right)
=\sum_{n=1}^{\infty}  \frac{a(n) \chi_{j}(n)  }{n^s}\,.
$$
Since $0 \leq \Lambda_X(n) \leq \Lambda(n)$, we see that
for $\s>1$
\begin{align*}
\sum_{n=1}^{\infty}  \frac{a(n)  }{n^{\s}} =&
\exp \left( \sum_{n \leq X^2}
\frac{\Lambda_X(n) \, }{n^{\s} \log n}  \right) \\
\leq & \exp \left( \sum_{n = 2}^{\infty}
\frac{\Lambda(n) \, }{n^{\s} \log n}  \right)
=\zeta(\s)\,.
\end{align*}
In particular,  $0 \leq a(n) \leq 1$ and $a(1)=1$. 

Next write
$$
\mathcal{P}_X(s)  =  \sum_{j=1}^{J} b_j e^{i \a_j} P_X(s, \chi_{j})
= \sum_{n=1}^{\infty}  \frac{a(n)}{n^s}
\left(\sum_{j=1}^{J} b_j e^{i \a_j}  \chi_{j}(n)  \right)
= \sum_{n=1}^{\infty}  \frac{a(n)B(n)}{n^s} \,,
$$
say, and assume from now on that $B(1)\neq 0$. (If $B(1)=0$, 
the following argument  would have to be modified slightly, and the  
number of zeros would change by $O(t)$.)
Setting
$B=\sum_{j=1}^{J} |b_j| $, we have  $|B(n)| \leq B$ for every $n$,
and there exists a positive constant $c_1$ and a real number $\omega$
such that $ B(1) = c_1 e^{i\omega}B$. It follows that
for $\s>1$ 
$$
\Re \,(e^{-i\omega}\mathcal{P}_X(s) ) \geq   c_1B - B\sum_{n=2}^{\infty}  \frac{1}{n^{\s}}
\geq B\left(c_1 - \int_{1}^{\infty} x^{-\s}   \right)
 = B \left(c_1 - \frac{1}{ \s-1}   \right)\,.
$$
This is positive if $\s > 1+1/c_1$. Thus, if $\s_{0}$ meets this condition,
$\Re \,(e^{-i\omega}\mathcal{P}_X(\s_0 +it) ) > 0$ for all $t$, and 
$\arg \mathcal{P}_X(\s_0 +it)$ varies by at most $\pi$ on   
$[\s_0, \s_0+it]$. 
It follows that $|\arg \mathcal{P}_X(\frac12+it)|$ is   less than or equal to 
the change in argument of $\mathcal{P}_X(s)$ on the segment $[\frac12+it, \s_0+it]$
plus $\pi$. By a well known lemma in Section 9.4 of Titchmarsh~\cite{T}, 
if $|\mathcal{P}_X(\s^{'}+it^{'})| \leq M(\s, t)$ for $\frac12 \leq \s \leq \s^{'}, 1\leq t^{'}\leq t$, then
this change in argument is 
$\ll_{\epsilon} \log (M(\frac12-\e, \t)/|\Re\,e^{-i\omega}\mathcal{P}_X(\s_0) |)+1$
for any $\epsilon>0$. Now 
$$
|\mathcal{P}_X(s)|  \leq  B\sum_{j=1}^{J}| P_X(s, \chi_{j})|
$$
and 
$$
 P_X(s, \chi_{j}) \ll  \exp \left( \sum_{n \leq X^2}
\frac{\Lambda_X(n) }{n^{\s} \log n}  \right)
\ll \exp \left( \frac{X^{2(1-\s)}}{\log X}   \right) \,.
$$
Thus, 
$$
\arg \mathcal{P}_X(\frac12+it) \ll_{\epsilon} \frac{X^{1+2\e}}{\log X} \,.
$$
This is a very crude bound but it suffices here.
By \eqref{zero lower bd 2}, we now have
$$
\mathcal{N}_X(t) \geq    \frac{t}{2\pi}\log \frac{tq}{2\pi} - \frac{t}{2\pi} 
  +O_{\epsilon}(X^{1+2\e}) \,,
$$
and the leading term is larger than the $O$-term if $X < t^{1-2\epsilon}.$
To leading order this is also the lower bound for the number of zeros of  
each $L_X(s, \chi)$. 
With more work we could  show  unconditionally that  when
$\log X/\log \t =o(1)$, the number of zeros arising from case $(2)$
is in fact   $=(1+o(1))(t/2\pi)\log(t/2\pi)$.

An analysis of the contribution of zeros from case $(1)$ is rather elaborate and we will
not attempt it here. One expects relatively few zeros to arise in this way, though,
because it is unlikely that the curve $z=\mathcal{P}_X(\frac12+it)$ will pass through the origin.
As with our previous results, the difficulty we have is not to prove that
there are lots of zeros on the line, but that there are not too many, 
and, just as before, we have only limited success with this.

The main point I wished  to illustrate here is that one can see immediately from the 
structure of $L_X(s, \chi)$ why  one might expect $100\%$ of the zeros of  linear 
combinations  of such functions to lie on the critical line.    
It therefore  suggests a reason this should  be true for 
linear combinations of actual L-functions, and  this reason is different from the  usual one.

 
 
 
 
 
\section{Appendix}

\begin{thm*}\label{thm on Lindelof 1} 
A  necessary and  sufficient condition for the truth of the Lindel\"{o}f 
Hypothesis is that for $\frac12 \leq \s \leq 2 $,
 $|s-1|> \frac{1}{10}$,  and $2\leq X \leq \t^2$,
\begin{equation}\label{Lindelof approx 4}
 \zeta(s) = \sum_{n \leq X} \frac{1}{n^s}   
+  O\left(  X^{ \frac12 - \sigma}\t^{\epsilon}   \right) \,.
\end{equation} 
Moreover, if the Riemann Hypothesis is true, then 
there exists a positive constant $C_1$ such that for $X$ and $s$
as above, 
\begin{equation}\label{Riemann approx 10}
 \zeta(s) = \sum_{n \leq X} \frac{1}{n^s}   
+  O\left(  X^{\frac12 - \sigma } e^{C_1\Phi(t)} \right)   \,.
 \end{equation}
 Here $\Phi(t)$ is an admissible function in the sense of Section~\ref{product approx}.
 In particular, we have 
  \begin{equation}\label{zeta bd}
\zeta(s) \ll e^{C_1\Phi(t)} 
\end{equation}
for $\frac12 \leq \s \leq 2 $ and
 $|s-1|> \frac{1}{10}$ 
\end{thm*}

 \begin{proof} The proof of a statement similar to the first assertion may 
 be found in  Titchmarsh~\cite{T}  (Theorem 13.3).  Moreover, the more difficult
implication (Lindel\"{o}f implies \eqref{Lindelof approx 4}) is proved by an easy 
modification of the proof of the second assertion, which we turn to now.

We apply   Perron's formula (Lemma 3.19 of Titchmarsh~\cite{T})   to 
$  \zeta(s)$  
and  obtain
\begin{equation}\label{sum 1 H}   %
\sum_{n \leq X}\frac{1}{n^s} =
\frac{1}{2\pi i} \int_{c-i U}^{c+i U}  \zeta(s+w) \frac{X^w}{w} \d w
+O\left( \frac{X^\frac12 \log 2X}{U}   \right)  
+O\left(X^{-\sigma} \right)  \,,
\end{equation}
where $X \geq 2$ and $c = \frac12 + \frac{1}{\log X}$.  
Letting $b = \frac12-\sigma -\frac{1}{\log \t}$ and 
 $\R$ the positively oriented rectangle with vertices $b \pm i U$ and 
$c \pm i U$, we find that
\begin{equation}\label{2H}
\frac{1}{2\pi i} \int_{\R}  \zeta(s+w) \frac{X^w}{w} \d w
= \zeta(s)+ \frac{X^{1-s} }{1-s}\,.
\end{equation}
On the Riemann Hypothesis,
\begin{equation}\label{3H}
\left(  \int_{b - i U}^{c- i U}  
+ \int_{c + i U}^{b + i U} \right)
\zeta(s+w) \frac{X^w}{w} \d w \ll  X^\frac12  e^{\Phi(U+\t)} U^{-1}\,.
\end{equation}
Here we have used the functional equation and the
 estimates $ |\zeta(\frac12 +\frac{1}{\log \t} +i (t+v))|  \ll e^{\Phi(U+\t)}$ 
 and $X^c \ll X^\frac12$.
Also by the Riemann Hypothesis,
\begin{align}\label{4H}
  \int_{b - i U}^{b + i U}  
\zeta(s+w) \frac{X^w}{w} \d w 
& \ll     X^{\frac12 - \sigma} e^{\Phi(U+\t)}     \notag
 \int_{0}^{U}  \left( b^2 + (v)^2 \right)^{-\frac12}  dv   \\
& \ll  X^{\frac12 - \sigma} e^{\Phi(U+\t)}  \log (U/b) \\
& \ll  X^{\frac12 - \sigma} e^{\Phi(U+\t)}  \log U   \,.   \notag
\end{align}
Combining \eqref{sum 1 H} - \eqref{4H},  
we obtain
\begin{equation*}\label{Basic Lindelof formula}  
\zeta(s) = \sum_{n \leq X}\frac{1}{n^s}  +  \frac{X^{1-s} }{s-1} 
+ O\left( \frac{X^{\frac12}  e^{\Phi(U+\t)} }{U}  \right)  
+  O\left(   X^{\frac12 - \sigma}  e^{\Phi(U+\t)} \log \t \right) \,.
\end{equation*}
Since  $X \leq \t^{2}$   
the second term on the right is $\ll X^{\frac12 - \sigma} (X^{\frac12}/  \t)
\ll X^{\frac12 - \sigma} $.
The third is $\ll X^{\frac12}  e^{\Phi(U+\t)} 
U^{-1} \ll X^{\frac12-\s}  e^{\Phi(U+\t)} $
since $U=\t+X^{\s}> X^\s$.
Thus,  we find that
\begin{equation*}
\zeta(s) = \sum_{n \leq X} \frac{1}{n^s}   
+  O\left(  X^{\frac12 - \sigma } e^{\Phi(U+\t)} 
\log  \t  \right)   \,. 
\end{equation*}
Finally, by \eqref{Phi ineq} and the fact that $\Phi$ is increasing, we have 
$\Phi(U+\t) \leq \Phi(\t^4 +2\t) \leq \Phi(2\t^4) \leq C_1\Phi(t)$.
This establishes \eqref{Riemann approx 10}. 

The bound in \eqref{zeta bd} follows immediately on taking $X=2$ in 
\eqref{Riemann approx 10}. 

\end{proof}

Now  set
$$
S(u)=  \sum_{n \leq u} \frac{1}{n^{\frac12 + i t}} \,.
$$
Since $\zeta(\frac12+it) \ll e^{\Phi(t)}$, 
by \eqref{Riemann approx 10} we see that
\begin{equation*}
S(u) \ll    e^{C_1\Phi(t)} 
\end{equation*}
for $1 \leq u  \leq \t^{2 }$.
By Stieltjes integration,  if
$\sigma < \frac12$, 
\begin{align*}
 \sum_{n \leq X} \frac{1}{n^{\sigma  + i t}} 
&  = \int_{1^{-}}^{ X} u^{  \frac12 -  \sigma}\, d S(u)   
  =  u^{  \frac12 -  \sigma} S(u) \bigg|_{1^{-}}^{ X} -
( \tfrac12 -  \sigma) \int_{1}^{ X}                                      \notag
 u^{  -\frac12 -  \sigma  } S(u)    \, d u \\
& \ll  X^{  \tfrac12 -  \sigma} e^{C_1\Phi(t)} +		 \notag
( \tfrac12 -  \sigma )  e^{A_1\Phi(t)} \int_{1}^{ X}  u^{ - \frac12 -  \sigma  }   \, d u\\
& \ll  X^{  \frac12 -  \sigma} e^{C_1\Phi(t)} \,.           \notag
\end{align*}
We also have from  \eqref{Riemann approx 4} and
\eqref{zeta bd} that when $\sigma \geq  \frac12$
\begin{equation*}
 \sum_{n \leq X} \frac{1}{n^{\sigma  + i t}}  \ll     e^{C_1\Phi(t)} \,.
\end{equation*}
Combining our estimates, we obtain the

\begin{cor*}\label{bound for sums}  
Let   $  1 \leq X  \leq \t^{2}$,  $|\sigma| \leq 2$,  and $|s-1|> \frac{1}{10}$.
If the Riemann Hypothesis is true we have 
\begin{equation*}\label{Riemann approx 5}
  \sum_{n \leq X} \frac{1}{n^s}   \ll
   X^{\max ( \frac 12 - \sigma,\, 0 ) } \,e^{C_1\Phi(t)} \,.
\end{equation*}
Moreover, 
A necessary and 
sufficient condition for the truth of the Lindel\"{o}f Hypothesis is that   
\begin{equation*}\label{Lindelof approx 5}
  \sum_{n \leq X} \frac{1}{n^s}   \ll
   X^{\max ( \frac 12 - \sigma,\, 0 ) } \,\t^{\epsilon} \,.
\end{equation*}
\end{cor*}


\end{document}